\documentclass{article}
\usepackage[margin=2.5cm]{geometry}
\usepackage{ucs}
\usepackage{dsfont}
\usepackage{mathrsfs}
\usepackage{mathtools}
\usepackage{amsmath}
\usepackage{amsfonts}
\usepackage{amssymb}
\usepackage{soul}
\usepackage[makeroom]{cancel}
\usepackage[english]{babel}
\usepackage{ucs}
\usepackage[colorlinks=true,linkcolor=blue,citecolor=red]{hyperref}
\usepackage{wasysym}
\usepackage{marvosym}
\usepackage{verbatim}
\usepackage{multirow}
\usepackage{bussproofs}
\usepackage{setspace}
\usepackage{graphicx}
\usepackage{amsthm}
\usepackage{forest}
\forestset{smullyan tableaux/.style={for tree={math content},where n children=1{!1.before computing xy={l=\baselineskip},!1.no edge}{},closed/.style={label=below:$\times$},},}
\usepackage{longtable}
\usepackage{authblk}
\usepackage[all]{xy}
\usepackage{tikz}
\usetikzlibrary{arrows,calc,patterns,positioning,shapes}
\usetikzlibrary{decorations.pathmorphing}
\tikzset{
modal/.style={>=stealth',shorten >=1pt,shorten <=1pt,auto,
node distance=1.5cm,semithick},
world/.style={circle,draw,minimum size=1cm},
point/.style={circle,draw,fill=black,inner sep=0.5mm},
reflexive/.style={->,in=120,out=60,loop,looseness=#1},
reflexive/.default={5},
reflexive point/.style={->,in=135,out=45,loop,looseness=#1},
reflexive point/.default={25},
}
\newcommand{\Ltriangle}{\mathcal{L}_\blacktriangle}
\newcommand{\Lbox}{\mathcal{L}_\Box}
\newtheorem{theorem}{Theorem}

\newtheorem{corollary}{Corollary}
\newtheorem{lemma}{Lemma}

\theoremstyle{definition}
\newtheorem{definition}{Definition}
\theoremstyle{remark}
\newtheorem{remark}{Remark}
\newtheorem{example}{Example}
\newtheorem{convention}{Convention}
\newcommand{\AFDE}{\mathbf{K}^\blacktriangle_\mathbf{FDE}}
\title{Non-contingency in a~paraconsistent setting}
\begin{document}
\providecommand{\keywords}[1]{\small\textbf{Keywords: } #1}
\title{Non-contingecy in a~paraconsistent setting\thanks{The research of the first author was funded by the grant ANR JCJC 2019, project PRELAP (ANR-19-CE48-0006). The authors wish to thank two anonymous reviewers for their helpful comments and remarks.\\\indent This is a~preprint version of the following paper:~\href{https://doi.org/10.1093/jigpal/jzac081}{doi: 10.1093/jigpal/jzac081}.}}
\author[1]{Daniil Kozhemiachenko}
\affil[1]{INSA Centre Val de Loire, Univ.\ Orl\'{e}ans, LIFO EA 4022, France\\\href{mailto:daniil.kozhemiachenko@insa-cvl.fr}{daniil.kozhemiachenko@insa-cvl.fr}}
\author[2]{Liubov Vashentseva}
\affil[2]{Department of Logic, Faculty of Philosophy, Lomonosov Moscow State University, Moscow 119991, Russia\\\href{mailto:vashentsevaliubov@gmail.com}{vashentsevaliubov@gmail.com}}
\maketitle
\begin{abstract}
We study an extension of First Degree Entailment (FDE) by Dunn and Belnap with a~non-contingency operator $\blacktriangle\phi$ which is construed as ‘$\phi$ has the same value in all accessible states’ or ‘all sources give the same information on the truth value of $\phi$’. We equip this logic dubbed $\AFDE$ with frame semantics and show how the bi-valued models can be interpreted as interconnected networks of Belnapian databases with the $\blacktriangle$ operator modelling search for inconsistencies in the provided information. We construct an analytic cut system for the logic and show its soundness and completeness. We prove that $\blacktriangle$ is not definable via the necessity modality $\Box$ of $\mathbf{K_{FDE}}$. Furthermore, we prove that in contrast to the classical non-contingency logic, reflexive, $\mathbf{S4}$, and $\mathbf{S5}$ (among others) frames \emph{are definable}.

\vspace{.5em}

\noindent
\end{abstract}
\maketitle
\begin{spacing}{1}
\section{Introduction\label{introduction}}
\subsection{Classical logics of (non-)contingency}
Logics of (non-)contingency extend the language of propositional logic with the operator $\triangle$. If we employ Kripke semantics, $\triangle\phi$ is then considered to be true at some state $w$ iff
\begin{itemize}
\item $\phi$ is true in all accessible states or false in all accessible states;
\item $\phi$ \emph{has the same truth value} in all accessible states.
\end{itemize}
Note that these two conditions are equivalent in classical logic.

Depending on the intended interpretation, $\triangle\phi$ can be understood as ‘$\phi$ is not contingent’, ‘I~know whether $\phi$ is true’ (in the epistemic context), or ‘I have an opinion regarding the truth-value of $\phi$’ (in the doxastic context). Classical non-contingency logics --- in particular, their proof theory and model theory --- have been attracting attention for quite some time~\cite{Humberstone1995,Kuhn1995,Zolin1999,Zolin2002,Humberstone2013,Costa-Leite2016,Fan2019}. They are well motivated and can be applied to solve several epistemic puzzles (cf.~\cite[\S\S1,8]{FanWangvanDitmarsch2015} for more details regarding the use of the ‘knowing whether’ operator).

However, non-contingency logics suffer from an unfortunate drawback. It is known that numerous useful first-order properties on Kripke frames can be elegantly defined via modal formulas with $\Box$. On the other hand, many of them cannot be defined using $\triangle$-formulas. This is due to results by Zolin~\cite[Lemma~4.6]{Zolin1999} that every $\triangle$-definable class of frames contains the class of
partial-functional frames\footnote{I.e., frames where each state has at most one successor. Zolin calls such frames ‘functional’.}. Even more so, all partial-functional frames are equivalent w.r.t.\ classical $\triangle$-formulas as shown in~\cite[Proposition 3.7]{FanWangvanDitmarsch2015}.

A straightforward corollary is that (among others) serial, reflexive, symmetric, transitive, and Euclidean frames are not definable with $\triangle$-formulas, although they are definable with $\Box$-formulas.
\subsection{Modal logics based upon First Degree Entatilment and its relatives}\label{subsec:modalfde}
First Degree Entailment ($\mathbf{FDE}$) is a~paraconsistent logic over the $\{\neg,\wedge,\vee\}$ language formulated by Dunn and Belnap in a~series of papers~\cite{Dunn1976,Belnap1977computer,Belnap1977fourvalued}. One of its main ideas was to retain classical truth and falsity conditions but treat the truth and falsity of propositions independently. In particular,
\begin{center}
\begin{tabular}{c|c|c}
&\textbf{is true when}&\textbf{is false when}\\\hline
$\neg\phi$&$\phi$ is false&$\phi$ is true\\
$\phi_1\wedge\phi_2$&$\phi_1$ and $\phi_2$ are true&$\phi_1$ is false or $\phi_2$ is false\\
$\phi_1\vee\phi_2$&$\phi_1$ is true or $\phi_2$ is true&$\phi_1$ and $\phi_2$ are false
\end{tabular}
\end{center}
Thus, any proposition $\phi$ could be not only true or false but also have both values (i.e., a~truth value ‘glut’ --- both true and false) or have no value (a truth value ‘gap’ --- neither true nor false). This is why there are no theorems in the $\mathbf{FDE}$. However, sequents of the form $\phi\vdash\chi$ (‘first-degree entailments’, whence the name of the logic, or ‘formula-formula’ sequents) where $\phi$ and $\chi$ are formulas in the $\{\neg,\wedge,\vee\}$ language are valid if whenever $\phi$ is true, then $\chi$ is true too. An equivalent notion of validity could be formulated via the preservation of non-falsity: if $\phi$ is not false, then so is $\chi$.
\begin{remark}
In~\cite{Belnap1977fourvalued,Belnap1977computer}, $\mathbf{FDE}$ is formulated as a~four-valued logic with truth table semantics where each value from $\{\mathbf{T},\mathbf{F},\mathbf{B},\mathbf{N}\}$ represents what a~computer or a~database might be told regarding a~given statement.
\begin{itemize}
\item $\mathbf{T}$ stands for ‘just told True’.
\item $\mathbf{F}$ stands for ‘just told False’.
\item $\mathbf{B}$ (or \textbf{Both}) stands for ‘told both True and False’.
\item $\mathbf{N}$ (or \textbf{None}) stands for ‘told neither True nor False’.
\end{itemize}
\end{remark}

$\mathbf{FDE}$ has well-studied modal expansions (cf.~e.g.~\cite{Priest2008FromIftoIs,Priest2008,OdintsovWansing2017,Drobyshevich2020}). They usually employ frame semantics and use either Hilbert-style or tableaux calculi for their proof theory. There is also work on the correspondence theory for expansions of $\mathbf{FDE}$ with $\Box$ modality and (or) some implication (cf., e.g.~\cite{RivieccioJungJansana2017,Drobyshevich2020}).

To the best of our knowledge, however, there is no work done on the expansions of $\mathbf{FDE}$ with (non-)contingency modalities. Thus, there is a~gap between the classical logic on the one hand and $\mathbf{FDE}$ on the other. Many frame properties are not definable with the classical $\triangle$-formulas and there is no work done on paraconsistent non-contingency logics. In this paper, we try to fill in this gap.
\subsection{Motivation and plan of the paper}
Our motivation and our goal thus come from two sources. The first one is the classical non-contingency logic. The second one is modal expansions of the First Degree Entailment.

We are going to introduce an expansion of $\mathbf{FDE}$ dubbed $\AFDE$ with the non-contingency modality $\blacktriangle\phi$\footnote{We reserve $\triangle$ for the classical non-contingency operator so as to avoid confusion.} which we will informally interpret as ‘the agent knows the truth value of $\phi$’ following~\cite{FanWangvanDitmarsch2015}, ‘the truth value of $\phi$ is the same in all accessible states’, or ‘all available sources give the same information regarding $\phi$’. We will as well show that some of the frame conditions undefinable with classical $\triangle$-formulas are in fact definable with ‘formula-formula’ $\AFDE$ sequents. Thus, we will mend the gap mentioned above.

The remainder of the paper is structured as follows. In \S\ref{sec:languageandsemantics}, we present the language which we call $\Ltriangle$ as well as semantics for the expansion of $\mathbf{FDE}$ with the non-contingency modality. We motivate our semantics for $\blacktriangle$ and provide several contexts in which our semantics can be used.

In \S\ref{sec:proofsystem}, we present an analytic cut system for $\AFDE$ and then show its soundness and completeness. As a~corollary of completeness, we obtain the subformula property.

In \S\ref{sec:expressivity}, we deal with the expressivity of $\Ltriangle$. In particular, we prove that, in contrast to the classical non-contingency logic, $\blacktriangle$ cannot be defined using $\Box$ from $\mathbf{K_{FDE}}$. Neither can $\Box$ be defined via $\blacktriangle$.

In \S\ref{sec:framedefinability}, we prove the definability of several frame classes via finite sets of $\AFDE$ sequents. In particular, we show that reflexive ($\mathbf{T}$) and preordered ($\mathbf{S4}$) frames as well as the frames whose accessibility relation is an equivalence relation ($\mathbf{S5}$ frames) are definable in contrast to classical non-contingency logic.

Finally, in~\S\ref{sec:conclusion}, we recapitulate our results and set the goals for future research.
\section{Language and semantical framework\label{sec:languageandsemantics}}
The formulas of $\Ltriangle$ are built from the countable set of propositional variables $\mathsf{Var}=\{p,q,r,\ldots\}$ according to the following grammar in Backus--Naur form:
\[\phi\!\coloneqq\!p\in\mathsf{Var}\mid\neg\phi\mid\phi\wedge\phi\mid\phi\vee\phi\mid\blacktriangle\phi\]
We will denote the set of variables occurring in $\phi$ via $\mathsf{Var}(\phi)$.
\subsection{Interpretation of connectives}\label{subsec:interpretationofconnectives}
We follow Odintsov's and Wansing's~\cite{OdintsovWansing2010,OdintsovWansing2017} presentation of semantics of non-classical modal logics which uses \emph{two} valuations on a~frame --- $v^+$ (support of truth) and $v^-$ (support of falsity). Note however, that it is possible to produce an equivalent semantics based on models with one valuation (assigning one value from $\{\mathbf{T},\mathbf{B},\mathbf{N},\mathbf{F}\}$) as done by Priest~\cite{Priest2008FromIftoIs,Priest2008}.
\begin{definition}[Semantics]\label{def:AFDEsemantics}
A \emph{frame} is a~tuple $\mathfrak{F}=\langle W,R\rangle$ with $W\neq\varnothing$, $R$ being a~binary accessibility relation on $W$. A~\emph{model} is a~tuple $\mathfrak{M}=\langle W,R,v^+,v^-\rangle$ with $\langle W,R\rangle$ being a~frame and $v^+$ and $v^-$~being maps from $\mathsf{Var}$ to $2^W$ interpreted as support of truth and support of falsity, respectively. If $w\in\mathfrak{M}$, a~tuple $\langle\mathfrak{M},w\rangle$ is called a~\emph{pointed model}.

The semantics of propositional formulas is defined inductively as usual.
\begin{itemize}
\item $\begin{matrix}\mathfrak{M},w\vDash^+p&\Leftrightarrow&w\in v^+(p)\\\mathfrak{M},w\vDash^-p&\Leftrightarrow&w\in v^-(p)\end{matrix}$
\item $\begin{matrix}\mathfrak{M},w\vDash^+\neg\phi&\Leftrightarrow&\mathfrak{M},w\vDash^-\phi\\\mathfrak{M},w\vDash^-\neg\phi&\Leftrightarrow&\mathfrak{M},w\vDash^+\phi\end{matrix}$
\item $\begin{matrix}\mathfrak{M},w\vDash^+\phi_1\wedge\phi_2&\Leftrightarrow&\mathfrak{M},w\vDash^+\phi_1\text{ and }\mathfrak{M},w\vDash^+\phi_2\\\mathfrak{M},w\vDash^-\phi_1\wedge\phi_2&\Leftrightarrow&\mathfrak{M},w\vDash^-\phi_1\text{ or }\mathfrak{M},w\vDash^-\phi_2\end{matrix}$
\item $\begin{matrix}\mathfrak{M},w\vDash^+\phi_1\vee\phi_2&\Leftrightarrow&\mathfrak{M},w\vDash^+\phi_1\text{ or }\mathfrak{M},w\vDash^+\phi_2\\\mathfrak{M},w\vDash^-\phi_1\vee\phi_2&\Leftrightarrow&\mathfrak{M},w\vDash^-\phi_1\text{ and }\mathfrak{M},w\vDash^-\phi_2\end{matrix}$
\end{itemize}
To make the presentation of the semantics for $\blacktriangle$ more concise we introduce the following conditions.
\begin{equation}
\tag{$t_1\blacktriangle$}\label{t1conditionI}
\begin{array}{rc}
\forall w_1,\!w_2\!:\!R(w_0,w_1)~\&~R(w_0,w_2)&\Rightarrow\\(\mathfrak{M},w_1\!\vDash^+\!\phi\!\Rightarrow\!\mathfrak{M},w_2\!\vDash^+\!\phi)~\&~(\mathfrak{M},w_1\!\vDash^-\!\phi\!\Rightarrow\!\mathfrak{M},w_2\vDash^-\phi)
\end{array}
\end{equation}
\begin{equation}
\tag{$t_2\blacktriangle$}\label{t2conditionI}
\forall w_1:R(w_0,w_1)\Rightarrow\mathfrak{M},w_1\vDash^+\phi\text{ or }\mathfrak{M},w_1\vDash^-\phi
\end{equation}
\begin{equation}
\tag{$f_1\blacktriangle$}\label{f1conditionI}
\exists w_1,w_2:R(w_0,w_1)~\&~R(w_0,w_2)~\&~\mathfrak{M},w_1\vDash^+\phi~\&~\mathfrak{M},w_2\nvDash^+\phi
\end{equation}
\begin{equation}
\tag{$f_2\blacktriangle$}\label{f2conditionI}
\exists w_1,w_2:R(w_0,w_1)~\&~R(w_0,w_2)~\&~\mathfrak{M},w_1\vDash^-\phi~\&~\mathfrak{M},w_2\nvDash^-\phi
\end{equation}
\begin{equation}
\tag{$f_3\blacktriangle$}\label{fconditionS}
\exists w_1,w_2:R(w_0,w_1)~\&~R(w_0,w_2)~\&~\mathfrak{M},w_1\vDash^+\phi~\&~\mathfrak{M},w_2\vDash^-\phi
\end{equation}

In light of these conditions, support of truth and support of falsity of $\blacktriangle$ is defined as follows.
\begin{itemize}
\item $\begin{matrix}\mathfrak{M},w_0\vDash^+\blacktriangle\phi&\Leftrightarrow&\eqref{t1conditionI}\text{ and }\eqref{t2conditionI}\\\mathfrak{M},w_0\vDash^-\blacktriangle\phi&\Leftrightarrow&\eqref{f1conditionI}\text{ or }\eqref{f2conditionI}\text{ or }\eqref{fconditionS}\end{matrix}$
\end{itemize}
\end{definition}

In what follows, we are going to use the following definition of validity via truth preservation\footnote{As we will see in Theorem~\ref{theorem:AFDEcontraposition}, we could equivalently define the validity via the \emph{non-falsity preservation}. However, it is customary to give definitions of validity and entailment via the truth preservation for extensions and expansions of $\mathbf{FDE}$.}.
\begin{definition}\label{def:validity}
Let $\mathfrak{F}$ be a~frame. $\phi\vdash\chi$ is \emph{valid on $\mathfrak{F}$} iff for any model $\mathfrak{M}$ on $\mathfrak{F}$, and for any $w\in\mathfrak{M}$, if $\mathfrak{M},w\vDash^+\phi$, then $\mathfrak{M},w\vDash^+\chi$.

$\phi\vdash\chi$ is \emph{valid} iff it is valid on every frame.
\end{definition}
\begin{remark}\label{rem:LETFcomparison1}
It is instructive to note\footnote{We are grateful to the handling editor for bringing this to our attention.} that the semantics of $\blacktriangle$ bears significant similarities to the semantics of the classicality operator $\circ$ and its dual non-classicality operator $\bullet$ of $\mathit{LET}_F$ as described in~\cite[Definitions~2 and~7]{AntunesCarnielliKapsnerRodriguez2020}. Furthermore, $\circ\phi$ is interpreted as ‘the information on $\phi$ is reliable’ which is also similar to how we interpret $\blacktriangle$ (cf.~Examples~\ref{ex:witnesses} and~\ref{ex:Belnapnetwork}).

However, there are several notable differences between $\AFDE$ on the one hand and $\mathit{LET}_F$ on the other. First of all, $\AFDE$ does not have valid formulas (cf.~Remark~\ref{rem:notautologies}) while $\mathit{LET}_F$ does: namely $\circ\phi\vee\bullet\phi$ is valid. Second, $\mathit{LET}_F$ presupposes that the accessibility relation on the frame is a partial order. In \S\ref{sec:framedefinability}, we will see that \emph{pre-ordered} ($\mathbf{S4}$) frames are definable in $\AFDE$ which will allow us to observe further differences between $\AFDE$ and $\mathit{LET}_F$ (cf.~Remark~\ref{rem:LETFcomparison2}).
\end{remark}
\begin{convention}
Let $\mathfrak{M}$ be a~model and let $w\in\mathfrak{M}$. We will use the following naming conventions.
\begin{center}
\begin{tabular}{r@{ --- }l}
$\mathfrak{M},w\vDash^+\phi$&$\phi$ \emph{is true at} $w$\\
$\mathfrak{M},w\vDash^-\phi$&$\phi$ \emph{is false at} $w$\\
$\mathfrak{M},w\nvDash^+\phi$&$\phi$ \emph{is not-true at} $w$\\
$\mathfrak{M},w\nvDash^-\phi$&$\phi$ \emph{is not-false at} $w$
\end{tabular}
\end{center}
In what follows, we will understand phrases such as ‘$\phi$ is true at $w$’ in Belnapian sense (i.e., as ‘$\phi$ is \emph{at least true} at $w$’), not in the classical sense (‘$\phi$ is \emph{true and not-false} at $w$’) unless specified otherwise.
\end{convention}
\begin{convention}[Notation in the models]
Throughout the paper, we are going to give examples of various models. In order to specify the values of variables in a~given state, we will use the following shorthands.
\begin{center}
\begin{tabular}{c|c}
\textbf{notation}&\textbf{meaning}\\\hline
$w:p^+$&$p$ is true and not-false at $w$\\
$w:p^-$&$p$ is false and not-true at $w$\\
$w:p^\pm$&$p$ is both true and false at $w$\\
$w:\xcancel{p}$&$p$ is neither true nor false at $w$
\end{tabular}
\end{center}
\end{convention}
\begin{remark}\label{rem:notautologies}
Note that just as in $\mathbf{FDE}$, there is no formula $\phi\in\Ltriangle$ s.t.
\begin{itemize}
\item for any pointed model $\langle\mathfrak{M},w\rangle$, $\mathfrak{M},w\vDash^+\phi$ or 
\item for any pointed model $\langle\mathfrak{M},w\rangle$, $\mathfrak{M},w\nvDash^-\phi$.
\end{itemize}
Indeed, consider the models in fig.~\ref{fig:Trivexample}.
\begin{figure}
\centering
\begin{tikzpicture}[modal,node distance=0.5cm,world/.append style={minimum
size=1cm}]
\node[world] (w) [label=below:{$w$}] {$p^\pm$};
\node[] [left=of w] {$\mathfrak{M}$:};
\path[->] (w) edge[reflexive] (w);
\end{tikzpicture}
\hfil
\begin{tikzpicture}[modal,node distance=0.5cm,world/.append style={minimum
size=1cm}]
\node[world] (w') [label=below:{$w'$}] {$\xcancel{p}$};
\node[] [left=of w'] {$\mathfrak{M}'$:};
\path[->] (w') edge[reflexive] (w');
\end{tikzpicture}
\caption{All variables have the same values exemplified by $p$.}
\label{fig:Trivexample}
\end{figure}
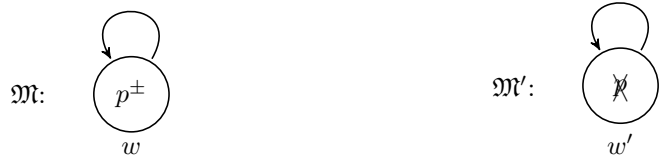
One can check that for any $\phi\in\Ltriangle$,
\begin{itemize}
\item $\mathfrak{M},w\vDash^+\phi$ and $\mathfrak{M},w\vDash^-\phi$;
\item $\mathfrak{M}',w'\nvDash^+\phi$ and $\mathfrak{M}',w'\nvDash^-\phi$. 
\end{itemize}
Thus, it makes sense to speak of valid \emph{sequents}, not formulas.
\end{remark}
\begin{convention}
For any state $w\in\mathfrak{M}$, we set $R(w)=\{w'\mid wRw'\}$.
\end{convention}
Let us now discuss the semantics for $\blacktriangle$ in more detail. Definition~\ref{def:AFDEsemantics} gives the following conditions on the Belnapian values of $\blacktriangle\phi$ in a~given state.
\begin{itemize}
\item[a.] $\blacktriangle\phi$ is \emph{true and not-false} at $w$ iff $\phi$ is either true and not-false in all accessible states or false and not-true in all accessible states.
\item[b.] $\blacktriangle\phi$ is \emph{both true and false} at $w$ iff $R(w)\neq\varnothing$ and $\phi$ is both true and false in all accessible states.
\item[c.] $\blacktriangle\phi$ is \emph{neither true nor false} at $w$ iff $R(w)\neq\varnothing$ and $\phi$ is neither true nor false in all accessible states.
\item[d.] $\blacktriangle\phi$ is \emph{false and not-true} at $w$ iff there are two accessible states such that $\phi$ has different truth values therein.
\end{itemize}

Recall first, that in classical logic $\triangle\phi$ (‘$\phi$ is non-contingent’) can be understood in two \emph{classically equivalent} ways:
\begin{itemize}
\item[(i.)] $\phi$ is true in all accessible states or false in all accessible states --- this interpretation comes from the reading of $\triangle\phi$ as being equivalent to $\Box\phi\vee\Box\neg\phi$;
\item[(ii.)] $\phi$ \emph{has the same truth value} in all accessible states.
\end{itemize}
In the case of $\mathbf{FDE}$, however, the second interpretation is stronger\footnote{In \S\ref{sec:expressivity}, we will see that a~straightforward expansion of $\mathbf{FDE}$ with $\Box$ cannot define $\blacktriangle$.} than the first. Indeed, if $\phi$ has the same truth value in all accessible states, then it is either true in all accessible states or false in all accessible states\footnote{It is possible that $\phi$ is neither true nor false in all accessible states but then it means that $\phi$ has \emph{no} value. Note that this argument can be formalised once the semantics for $\Box$ is given --- cf. \S\ref{sec:expressivity} for more details.}. On the contrary, it is possible for a~formula to be true in all accessible states and false in some just as in fig.~\ref{fig:trueandfalse}.
\begin{figure}
\begin{tikzpicture}[modal,node distance=1cm,world/.append style={minimum
size=1cm}]
\node[world] (w0) [label=below:{$w_0$}] {$p^+$};
\node[world] (w1) [right=of w0] [label=below:{$w_1$}] {$p^\pm$};
\node[] [left=of w0] {$\mathfrak{M}$:};
\path[->] (w0) edge[reflexive] (w0);
\path[->] (w0) edge (w1);
\end{tikzpicture}
\caption{$p$ is true and not-false at $w_0$ but is both true and false at $w_1$. Thus, $\blacktriangle p$ is false and not-true at $w_0$.}
\label{fig:trueandfalse}
\end{figure}
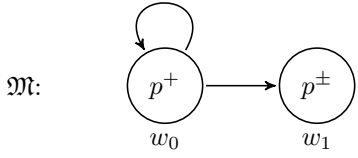
Thus, we cannot rightfully say that the value of $p$ is non-contigent in all states accessible from $w_0$. Moreover, in example~\ref{ex:witnesses}, we propose a~situation that explicitly shows that defining $\blacktriangle\phi$ as in (i.)\ might lead to an undesired conclusion. This means that only the interpretation in (ii.)\ corresponds to the intuition behind the non-contingency operator in the case of $\mathbf{FDE}$.

The d.\ case in the list above is well-aligned with the intuition of non-contingency. The choice of semantics for the cases a.--c., however, is less obvious and requires some explanation. To illustrate it better, we give contexts and examples that motivate the chosen semantics of $\blacktriangle\phi$ read as ‘the value of $\phi$ is the same in all accessible states’.
\paragraph{\textbf{Checking testimonies}}\label{par:witnesses}
Imagine that an investigator reads testimonies of several witnesses $\{w_1,\ldots,w_m\}$ that were asked to confirm or deny whether some events $\{p_1,\ldots,p_n\}$ happened or not.

In this setting, we can associate the testimonies with the states accessible to the investigator (whom we represent as $w$). Statements then take their values in these states in the expected fashion:
\begin{itemize}
\item if witness $w_i$ confirms that $p_j$ did happen and does not deny it, then $p_j$ is true and not-false at $w_i$;
\item if $w_i$ denies that $p_j$ happen, then $p_j$ is false and not-true at the corresponding state;
\item if $w_i$ gives a~contradictory account (which witnesses sometimes do) regarding $p_j$, e.g., first denies $p_j$ then confirms it, then $p_j$ is both true and false at $w_i$;
\item finally, if $w_i$ neither confirms nor denies $p_j$, $p_j$ is neither true nor false at~$w_i$.
\end{itemize}

The investigator does not know whether these events occurred and whether the testimonies are truthful. What they can, however, do is check for the ‘anomalies’ or irregularities in the testimonies. They are not only accounts of different claims regarding one statement (e.g., when one witness says that $p$ did happen but the other says that it did not), but also contradictory accounts of a~single witness (when they get confused in their testimony) or situations when one witness gives testimony regarding $p$ but the other does not. If no irregularities are detected, the testimonies may be considered trustworthy or at least pertaining to the same set of events.

This ‘anomaly-checking’ can be modelled with the $\blacktriangle$ operator in the following fashion. We represent the investigator with $w$ and set $wRw_i$ for all $w_i$'s representing the witnesses' accounts. Then, we have the following four options regarding the value of $\blacktriangle\phi$ (‘the accounts regarding $\phi$ do not contain anomalies’):
\begin{itemize}
\item[I.] if $\blacktriangle\phi$ is \emph{true} at $w$, it means that all witnesses give the same account regarding $\phi$;
\item[II.] if $\blacktriangle\phi$ is \emph{not-false} at $w$, then there are no two different accounts regarding $\phi$, and there are no contradictory accounts either\footnote{I.e., no one got confused in their testimony. Note that one can, likewise, view a~contradictory account of a~witness regarding an event as \emph{two different accounts}.};
\item[III.] if $\blacktriangle\phi$ is \emph{false} at $w$, it means that the accounts of at least two witnesses regarding $\phi$ are different or that somebody contradicts themselves;
\item[IV.] if $\blacktriangle\phi$ is \emph{not-true} at $w$, then either there are two different accounts regarding $\phi$ or there are no accounts at all.
\end{itemize}
Observe, that II.\ describes the situation when there are no irregularities in the above-given sense. Note as well that I.\ differs from II.\ because the latter option allows for a~situation when nobody gave any account regarding $\phi$ while I.\ requires that these accounts \emph{must} be given.

It is important to mention that an investigator might not be inclined to consider $\blacktriangle\phi$ \emph{true} (as opposed to \emph{not-false}) when no witness gives any account regarding $\phi$ for two reasons. First, the investigator themselves does not have any information regarding $\phi$ at all. Second, it usually goes against intuition to claim that ‘everybody gives the same account on $\phi$’ when in fact no account is given.

Likewise, if all witnesses say that $\phi$ is both true and false, it is reasonable to say that ‘the accounts on $\phi$ \emph{do not contain} anomalies’ is actually \emph{false}. On the other hand, it is the case that all witnesses \emph{agree} in their accounts on $\phi$. Thus, we can state that $\blacktriangle\phi$ is \emph{true and false}.

Finally, we wish to stress an important difference between the following two situations: (i) when no witness gives an account on $\phi$ and (ii) when some witnesses agree about $\phi$ (say, confirm and do not deny it), but others do not provide any testimony. As we said above, in (i), there are no irregularities but there are no accounts to compare, whence $\blacktriangle\phi$ is not-true and not-false. However, in (ii), there is an evident irregularity: some witnesses testify and others do not. This renders $\blacktriangle\phi$ false and not-true.

To further illustrate the reading given above, we propose the following example with an investigator.
\begin{example}\label{ex:witnesses}
Assume that our investigator is searching for a~suspect who is short ($s$) and armed with a~pistol~($p$). Moreover, the investigator needs the evidence to be supported by all witnesses. Two witnesses testified to a police officer regarding the same suspicious person spotted by them near a~bank at 6~am, the~12th of October this year.
\begin{itemize}
\item The account of $w_1$ mentions that the suspicious person was short and armed with a~pistol.
\item The account of $w_2$ was unfortunately badly written. Not only did the witness contradict themselves by first stating that the person they saw was short but then saying that that very same person was ‘tall as a~basketball player’ (i.e., not short), but it seems that the police officer forgot to ask the witness whether the person was armed.
\end{itemize}
The situation can be represented with the model in fig.~\ref{fig:suspects}. Here, the investigator does not have any information regarding the suspicious person in question, whence $p$ and $s$ are neither true nor false at $w$.
\begin{figure}
\centering
\begin{tikzpicture}[modal,node distance=1.5cm,world/.append style={minimum
size=1cm}]
\node[world] (w) [label=above:{$w$}] {$\xcancel{p}$, $\xcancel{s}$};
\node[world] (w1) [left=of w] [label=below:{$w_1$}] {$p^+$, $s^+$};
\node[world] (w2) [right=of w] [label=below:{$w_2$}] {$s^\pm$, $\xcancel{p}$};
\path[->] (w) edge (w1);
\path[->] (w) edge (w2);
\end{tikzpicture}
\caption{$w$ is the investigator; $w_1$ and $w_2$ stand for the accounts of the witnesses.}
\label{fig:suspects}
\end{figure}
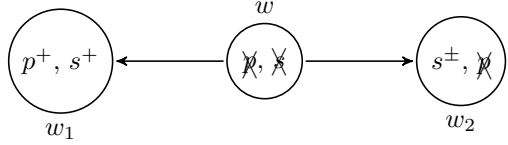

Observe that while $s$ is \emph{true} in all accessible states, and $p$ is \emph{not-false} in all of them, the accounts do contain anomalies w.r.t.\ both of them. Thus, $\blacktriangle p$ and $\blacktriangle s$ are \emph{false and not-true} at $w$. I.e., there are anomalies regarding $p$ and $s$ and the $w_1$ and $w_2$ give different testimonies on them. Indeed, it is easy to see with~$s$: while $s$ \emph{is true} at both $w_1$ and $w_2$, it is false only at $w_2$ which is an explicit anomaly: $w_2$'s account is self-contradictory and is different from that of $w_1$. Moreover, even though, there is no \emph{explicit} contradiction between $w_1$ and $w_2$ regarding $p$, their accounts are different: $w_1$ \emph{confirms} $p$ while $w_2$ \emph{does not}.
\end{example}

\paragraph{\textbf{A network of Belnapian computers}}
The next example is inspired by the ‘Belnapian computer’ from~\cite{Belnap1977computer}. Recall that in such a~computer (or database), each statement can be not only true or false but both true and false (e.g., if there was a~mistake in the input) and neither true nor false (if the input is incomplete) as well. Here, we propose to look at several databases connected to a~network which allows access from some databases to others.

This network is being audited by an external investigator who has access to all computers: i.e., if some accessible database refers to another one, then that one is accessible too\footnote{In other words, the accessibility relation can in some cases be transitive. It is also reasonable to assume that databases connected to a~network can refer not only to other databases but to themselves as well (thus, the relation is reflexive in this case). We will see that reflexive transitive frames, as well as reflexive frames, are definable in~\S\ref{sec:framedefinability}.\label{footnote:S4}}. Just as in the previous case with witness accounts, the investigator looks for inconsistencies in and between the databases. These can indicate that the records were falsified, or that the books are just badly kept.
\begin{example}\label{ex:Belnapnetwork}
An auditor examines a~database in the central office of a~stationery company which lists whether the goods are still in stock. The database tells that there are still pencils ($p$) and rulers ($r$) left at the store, and so does the database at the warehouse. But the database in the store says that the pencils are out of stock and does not contain any mention of rulers at all! Fortunately, the auditor was granted remote access to all databases to which the central one refers, and thus they can spot the irregularities in the bookkeeping.

The situation can be represented with fig.~\ref{fig:pencilsexample}. Evidently, $\blacktriangle p$ and $\blacktriangle r$ are \emph{false and not-true} at $w_c$.
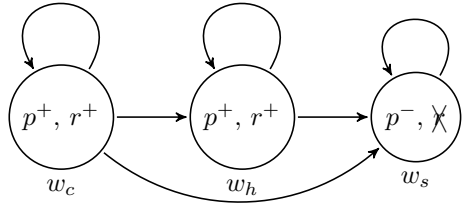
\begin{figure}
\centering
\begin{tikzpicture}[modal,node distance=1cm,world/.append style={minimum
size=1cm}]
\node[world] (wc) [label=below:{$w_c$}] {$p^+$, $r^+$};
\node[world] (wh) [right=of wc] [label=below:{$w_h$}] {$p^+$, $r^+$};
\node[world] (ws) [right=of wh] [label=below:{$w_s$}] {$p^-$, $\xcancel{r}$};
\path[->] (wc) edge[reflexive] (wc);
\path[->] (wh) edge[reflexive] (wh);
\path[->] (ws) edge[reflexive] (ws);
\path[->] (wc) edge (wh);
\path[->] (wh) edge (ws);
\path[->] (wc) edge[bend right=40] (ws);
\end{tikzpicture}
\caption{$w_c$ is the computer at the central office that the auditor is looking into; $w_h$ is the database at the warehouse, and $w_s$ is the database in the store.}
\label{fig:pencilsexample}
\end{figure}
Note that had the auditor not received access to $w_s$ (i.e., if $w_s$ had not been accessible from $w_c$), $\blacktriangle p$ and $\blacktriangle r$ would have been \emph{true and not-false} at $w_c$.
\end{example}
\begin{remark}
Another way of analysing networks of Belnapian computers is presented in the eponymous paper~\cite{ShramkoWansing2005}. There, the authors devise 16-valued logics that model the reasoning of a~central computer that collects information from the network. Our approach is different in that we take into account that the configuration of the network (represented via a~Kripke model) might be different from one case to another.

A related approach~\cite{Blasio2017} proposes logic $\mathbf{E}^B$ based upon $\mathbf{FDE}$ to analyse epistemic attitudes and formalise reasoning with acceptance and rejection. The paper also provides a sound and complete four-sided sequent calculus for $\mathbf{E}^B$.
\end{remark}
\paragraph{\textbf{$\blacktriangle\phi$ as ‘the value of $\phi$ is the same in all accessible states’}}
Finally, we argue that since the support of truth is thought to be independent of the support of falsity in $\mathbf{FDE}$, it is reasonable to demand that $\blacktriangle p$ is \emph{both true and false} when $p$ is both true and false in all accessible states even when we interpret it as ‘the value of $p$ is the same in all accessible states’. Likewise, we argue that $\blacktriangle p$ is \emph{neither true nor false} when $p$ is neither true nor false in all accessible states  (cf.\ fig.~\ref{fig:BandNfortriangle} for examples of models).
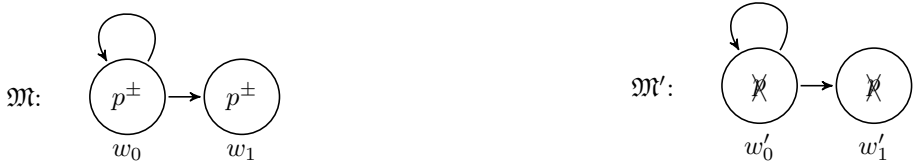
\begin{figure}
\begin{tikzpicture}[modal,node distance=0.5cm,world/.append style={minimum
size=1cm}]
\node[world] (w0) [label=below:{$w_0$}] {$p^\pm$};
\node[world] (w1) [right=of w0] [label=below:{$w_1$}] {$p^\pm$};
\node[] [left=of w0] {$\mathfrak{M}$:};
\path[->] (w0) edge[reflexive] (w0);
\path[->] (w0) edge (w1);
\end{tikzpicture}
\hfil
\begin{tikzpicture}[modal,node distance=0.5cm,world/.append style={minimum
size=1cm}]
\node[world] (w0) [label=below:{$w'_0$}] {$\xcancel{p}$};
\node[world] (w1) [right=of w0] [label=below:{$w'_1$}] {$\xcancel{p}$};
\node[] [left=of w0] {$\mathfrak{M}'$:};
\path[->] (w0) edge[reflexive] (w0);
\path[->] (w0) edge (w1);
\end{tikzpicture}
\caption{Here, $\blacktriangle p$ is \emph{both true and false} at $w_0$ and \emph{neither true nor false} at $w'_0$.}
\label{fig:BandNfortriangle}
\end{figure}

$\blacktriangle p$ is surely \emph{true} at $w_0$ since $p$ has the same value in all accessible states. But $\blacktriangle p$ is \emph{false} as well: $p$ is true at $w_0$ and false at $w_1$. So, we find ourselves in a~paradoxical situation: each source gives the same truth value to $\phi$ but since each source, in fact, gives \emph{two different truth values --- true and false ---} to $\phi$, the information the sources give is contradictory.

Likewise, $\blacktriangle p$ is \emph{not false} at $w'_0$: there are no values of $p$ to compare. But it is because of this that we may not be willing to say that $\blacktriangle p$ is \emph{true} at $w'_0$. For it is counterintuitive to claim that ‘$p$ has the same truth value in all accessible states’ if $p$ does not have any truth value at all in any of them. Indeed, this statement is \emph{vacuously} true from the classical point of view. But vacuously true statements (such as ‘all unicorns are green’ --- but there are no unicorns to speak of) do not correspond well to our intuition, and people do not tend to agree with them.
\subsection{Validity}
In~\S\ref{subsec:interpretationofconnectives}, we defined validity via truth preservation. However, the contexts given in examples~\ref{ex:witnesses} and~\ref{ex:Belnapnetwork} are closer to the definition of validity as non-falsity preservation. There, we were looking mostly for \emph{inconsistencies} while incomplete records were not a~huge problem as long as they were incomplete everywhere. On the other hand, an agent may be looking for some information that determines the truth value of a~statement and may tolerate that this information is contradictory if all sources agree on it. This approach suggests the definition of validity as truth preservation.

In the remainder of this section, we will show that these two definitions of validity are equivalent.
\begin{definition}[Dual models]\label{dualvaluationsdefinition}
For any model $\mathfrak{M}=\langle W,R,v^+,v^-\rangle$, we define its \emph{dual model} on the same frame $\mathfrak{M}_d=\langle W,R,v^+_d,v^-_d\rangle$ as follows.
\begin{align*}
\text{if }w\in v^+(p),w\notin v^-(p)&\text{ then }w\in v^+_d(p),w\notin v^-_d(p)\\
\text{if }w\in v^+(p),w\in v^-(p)&\text{ then }w\notin v^+_d(p),w\notin v^-_d(p)\\
\text{if }w\notin v^+(p),w\notin v^-(p)&\text{ then }w\in v^+_d(p),w\in v^-_d(p)\\
\text{if }w\notin v^+(p),w\in v^-(p)&\text{ then }w\notin v^+_d(p),w\in v^-_d(p)
\end{align*}
\end{definition}
In other words, if a~variable was either true and not-false or false and not-true in some state in a~model, then it remains such in the dual model. But if it was both true and false, it becomes neither true nor false and vice versa.
\begin{lemma}\label{lemma:dualvaluations}
Let $\mathfrak{M}=\langle W,R,v^+,v^-\rangle$ be a~model and $\mathfrak{M}_d=\langle W,R,v^+_d,v^-_d\rangle$ be its dual model. Then for any $\phi\in\Ltriangle$ and $w\in\mathfrak{M}$, it holds that
\begin{align*}
\text{if }\mathfrak{M},w\vDash^+\phi\text{ and }\mathfrak{M},w\nvDash^-\phi&\text{ then }\mathfrak{M}_d,w\vDash^+\phi\text{ and }\mathfrak{M}_d,w\nvDash^-\phi\\
\text{if }\mathfrak{M},w\vDash^+\phi\text{ and }\mathfrak{M},w\vDash^-\phi&\text{ then }\mathfrak{M}_d,w\nvDash^+\phi\text{ and }\mathfrak{M}_d,w\nvDash^-\phi\\
\text{if }\mathfrak{M},w\nvDash^+\phi\text{ and }\mathfrak{M},w\nvDash^-\phi&\text{ then }\mathfrak{M}_d,w\vDash^+\phi\text{ and }\mathfrak{M}_d,w\vDash^-\phi\\
\text{if }\mathfrak{M},w\nvDash^+\phi\text{ and }\mathfrak{M},w\vDash^-\phi&\text{ then }\mathfrak{M}_d,w\nvDash^+\phi\text{ and }\mathfrak{M}_d,w\vDash^-\phi
\end{align*}
\end{lemma}
\begin{proof}
We adapt the technique from~\cite{ZaitsevShramko2004english} and prove the statement by induction on $\phi$. The basis case of propositional variables holds by the construction of $v^+_d$ and $v^-_d$. The cases of propositional connectives hold by virtue of the admissibility of the contraposition in $\mathbf{FDE}$~\cite{Font1997,Dunn2000,ZaitsevShramko2004english}. It remains to consider the case of $\blacktriangle$.

Let $\phi=\blacktriangle\phi'$. If $\mathfrak{M},w\vDash^+\blacktriangle\phi'$ and $\mathfrak{M},w\nvDash^-\blacktriangle\phi'$, then either (1) $\mathfrak{M},w'\vDash^+\phi'$ and $\mathfrak{M},w'\nvDash^-\phi'$ for any accessible $w'$, or (2) $\mathfrak{M},w'\nvDash^+\phi'$ and $\mathfrak{M},w'\vDash^-\phi$ for any accessible $w'$. In the first case, by the induction hypothesis, we obtain that $\mathfrak{M}_d,w'\vDash^+\phi'$ and $\mathfrak{M}_d,w'\nvDash^-\phi'$ for any accessible $w'$. In the second case, we have $\mathfrak{M}_d,w'\nvDash^+\phi'$ and $\mathfrak{M}_d,w'\vDash^-\phi'$ for any accessible $w'$. In both cases, $\mathfrak{M}_d,w\vDash^+\blacktriangle\phi'$ and $\mathfrak{M}_d,w\nvDash^-\blacktriangle\phi'$.

Now let $\mathfrak{M},w\vDash^+\blacktriangle\phi'$ and $\mathfrak{M},w\vDash^-\blacktriangle\phi'$. Then $R(w)\neq\varnothing$ and $\mathfrak{M},w'\vDash^+\phi'$ and $\mathfrak{M},w'\vDash^-\phi'$ in all $w'\in R(w)$. By the induction hypothesis, $\mathfrak{M}_d,w'\nvDash^+\phi'$ and $\mathfrak{M}_d,w'\nvDash^-\phi'$ in any accessible $w'$ and thus $\mathfrak{M}_d,w\nvDash^+\blacktriangle\phi'$ and $\mathfrak{M}_d,w\nvDash^-\blacktriangle\phi'$.

The case of $\mathfrak{M},w\nvDash^+\blacktriangle\phi'$ and $\mathfrak{M},w\nvDash^-\blacktriangle\phi'$ can be proved in the same manner.

Lastly, if $\mathfrak{M},w\nvDash^+\blacktriangle\phi'$ and $\mathfrak{M},w\vDash^-\blacktriangle\phi'$, then there are two accessible states $w'$ and $w''$ s.t.\ one of the following options --- (a), (b), or (c) --- is the case.
\begin{enumerate}
\item[(a)] $\mathfrak{M},w'\vDash^+\phi'$ and $\mathfrak{M},w'\nvDash^-\phi'$ and
\begin{enumerate}
\item[(a.1)] $\mathfrak{M},w''\vDash^+\phi'$ and $\mathfrak{M},w''\vDash^-\phi'$, or
\item[(a.2)] $\mathfrak{M},w''\nvDash^+\phi'$ and $\mathfrak{M},w''\nvDash^-\phi'$, or
\item[(a.3)] $\mathfrak{M},w''\nvDash^+\phi'$ and $\mathfrak{M},w''\vDash^-\phi'$.
\end{enumerate}
\item[(b)] $\mathfrak{M},w'\nvDash^+\phi'$ and $\mathfrak{M},w'\vDash^-\phi'$ and
\begin{enumerate}
\item[(b.1)] $\mathfrak{M},w''\vDash^+\phi'$ and $\mathfrak{M},w''\vDash^-\phi'$ or
\item[(b.2)] $\mathfrak{M},w''\nvDash^+\phi'$ and $\mathfrak{M},w''\nvDash^-\phi'$.
\end{enumerate}
\item[(c)] $\mathfrak{M},w'\vDash^+\phi'$ and $\mathfrak{M},w'\vDash^-\phi'$ and $\mathfrak{M},w''\nvDash^+\phi'$ and $\mathfrak{M},w''\nvDash^-\phi'$.
\end{enumerate}
By the induction hypothesis, these are transformed as follows.
\begin{enumerate}
\item[(a)] $\mathfrak{M}_d,w'\vDash^+\phi'$ and $\mathfrak{M}_d,w'\nvDash^-\phi'$ and
\begin{enumerate}
\item[(a.1)] $\mathfrak{M}_d,w''\nvDash^+\phi'$ and $\mathfrak{M}_d,w''\nvDash^-\phi'$, or
\item[(a.2)] $\mathfrak{M}_d,w''\vDash^+\phi'$ and $\mathfrak{M}_d,w''\vDash^-\phi'$, or
\item[(a.3)] $\mathfrak{M}_d,w''\nvDash^+\phi'$ and $\mathfrak{M}_d,w''\vDash^-\phi'$.
\end{enumerate}
\item[(b)] $\mathfrak{M}_d,w'\nvDash^+\phi'$ and $\mathfrak{M}_d,w'\vDash^-\phi'$ and
\begin{enumerate}
\item[(b.1)] $\mathfrak{M}_d,w''\nvDash^+\phi'$ and $\mathfrak{M}_d,w''\nvDash^-\phi'$ or
\item[(b.2)] $\mathfrak{M}_d,w''\vDash^+\phi'$ and $\mathfrak{M}_d,w''\vDash^-\phi'$.
\end{enumerate}
\item[(c)] $\mathfrak{M}_d,w'\nvDash^+\phi'$ and $\mathfrak{M}_d,w'\nvDash^-\phi'$ and $\mathfrak{M}_d,w''\vDash^+\phi'$ and $\mathfrak{M}_d,w''\vDash^-\phi'$.
\end{enumerate}
Clearly, in all three cases, $\mathfrak{M}_d,w\nvDash^+\blacktriangle\phi$ and $\mathfrak{M}_d,w\vDash^-\blacktriangle\phi$, as required.
\end{proof}
\begin{theorem}\label{theorem:AFDEcontraposition}
$\phi\vdash\chi$ is valid on $\mathfrak{F}$ iff for any model $\mathfrak{M}$ on $\mathfrak{F}$ and for any $w\in\mathfrak{M}$, $\mathfrak{M},w\nvDash^-\phi$ implies $\mathfrak{M},w\nvDash^-\chi$.

In addition, the contraposition holds for $\AFDE$. That is, if $\phi\vdash\chi$ is valid, then $\neg\chi\vdash\neg\phi$ is valid.
\end{theorem}
\begin{proof}
Assume, there is a~model $\mathfrak{M}$ and $w\in\mathfrak{M}$ s.t.\ $\mathfrak{M},w\vDash^+\phi$ and $\mathfrak{M},w\nvDash^+\chi$, i.e. $\phi\vdash\chi$ is not valid.

We show that there exists a~model $\mathfrak{M}'$ and $w'\in\mathfrak{M}'$ s.t.\ $\mathfrak{M}',w'\nvDash^-\phi$ and $\mathfrak{M}',w'\vDash^-\chi$.

We have the following cases.
\begin{enumerate}
\item[A.] $\mathfrak{M},w\vDash^+\phi$ and $\mathfrak{M},w\nvDash^-\phi$ but $\mathfrak{M},w\nvDash^+\chi$ and $\mathfrak{M},w\vDash^-\chi$.
\item[B.] $\mathfrak{M},w\vDash^+\phi$ and $\mathfrak{M},w\nvDash^-\phi$ but $\mathfrak{M},w\nvDash^+\chi$ and $\mathfrak{M},w\nvDash^-\chi$
\item[C.] $\mathfrak{M},w\vDash^+\phi$ and $\mathfrak{M},w\vDash^-\phi$ but $\mathfrak{M},w\nvDash^+\chi$ and $\mathfrak{M},w\vDash^-\chi$.
\item[D.] $\mathfrak{M},w\vDash^+\phi$ and $\mathfrak{M},w\vDash^-\phi$ but $\mathfrak{M},w\nvDash^+\chi$ and $\mathfrak{M},w\nvDash^-\chi$.
\end{enumerate}
For A., the result follows immediately.

For B., C., and D., we use Lemma~\ref{lemma:dualvaluations} to build dual models where the following statements hold.
\begin{enumerate}
\item[B.] $\mathfrak{M}_d,w\vDash^+\phi$ and $\mathfrak{M}_d,w\nvDash^-\phi$ but $\mathfrak{M}_d,w\vDash^+\chi$ and $\mathfrak{M}_d,w\vDash^-\chi$
\item[C.] $\mathfrak{M}_d,w\nvDash^+\phi$ and $\mathfrak{M}_d,w\nvDash^-\phi$ but $\mathfrak{M}_d,w\nvDash^+\chi$ and $\mathfrak{M}_d,w\vDash^-\chi$.
\item[D.] $\mathfrak{M}_d,w\nvDash^+\phi$ and $\mathfrak{M}_d,w\nvDash^-\phi$ but $\mathfrak{M}_d,w\vDash^+\chi$ and $\mathfrak{M}_d,w\vDash^-\chi$.
\end{enumerate}

The converse direction can be shown in the same manner.

Assume now that $\phi\vdash\chi$ is valid but $\neg\chi\vdash\neg\phi$ is not. Then, there exist a~model $\mathfrak{M}$ and $w\in\mathfrak{M}$ s.t.\ $\mathfrak{M},w\vDash^+\neg\chi$ but $\mathfrak{M},w\nvDash^+\neg\phi$. Hence, $\mathfrak{M},w\vDash^-\chi$ and $\mathfrak{M},w\nvDash^-\phi$. Thus, $\phi\vdash\chi$ is not valid by the above-proven statement. A~contradiction.
\end{proof}
\section{Proof system\label{sec:proofsystem}}
In this section, we are presenting the proof system for our logic. We borrow the basic idea from the D'Agostino's $\mathbf{RE}_{\mathrm{fde}}$~\cite{DAgostino1990}. Namely, we define a~so-called analytic cut system --- a~modification of analytical tableaux that uses the ‘analytic cut’ rule which for the case of classical logic looks as follows:
\[\dfrac{}{\phi\mid\neg\phi}\]
for any formula $\phi$ being a~subformula of some formula on the branch.

We choose analytic cut systems for several reasons. First, they significantly reduce branching of the rules~(cf.~\cite{DAgostino1990,DAgostino1992,DAgostinoMondadori1994} for the classical logic and~\cite{CaleiroCarnielliConiglioMarcos2005,CaleiroMarcosVolpe2015,CaleiroMarcelinoRivieccio2018} for the non-classical ones) and simplify the structure of the derivations. Indeed, in our case, the semantics of $\blacktriangle$ would lead to tableaux rules with a~very complicated structure. Second, there are analytic cut systems for classical normal logics using the ‘necessity’ modality $\Box$ (cf., e.g.~\cite{Amerbauer1996}, \cite{Nguyen2001}, and~\cite{Indrzejczak2012}) as well as to non-classical logics~\cite{DAgostinoGabbay1994}. However, to the best of our knowledge, there are no analytic cut systems for non-classical logics with non-standard modalities.

Third, in contrast to natural deduction, sequent calculi, or Hilbert-style systems, it is usually straightforward to show the soundness and completeness of the analytic cut calculi. This is even more important since the completeness proofs for the modal extensions of $\mathbf{FDE}$ are prone to errors (cf.~\cite{Drobyshevich2020} for more details).
\subsection{Analytic cut}
We are going to use labelled formulas for our calculi. Since we have frame semantics, the label will consist of two parts: the generalised truth value assignment of the formula and the state where the formula has that truth value.
\begin{definition}\label{def:labelledformulas}
We fix a~countable set of state-labels $\mathsf{Lab}=\{w,w_0,w',\ldots\}$ and the~set of value-labels $\mathsf{Val}=\{\mathfrak{t},\mathfrak{f},\overline{\mathfrak{t}},\overline{\mathfrak{f}}\}$.

A \emph{labelled formula} is a~construction of the form $\mathsf{w}:\phi;\mathfrak{v}$ with $\phi\in\Ltriangle$, $\mathsf{w}\in\mathsf{Lab}$, and $\mathfrak{v}\in\mathsf{Val}$.
\end{definition}

The interpretations of labelled formulas are summarised in the following table.
\begin{center}
\begin{tabular}{c|c}
\textbf{Labelled formula}&\textbf{Interpretation}\\\hline
$w:\phi;\mathfrak{t}$&$\mathfrak{M},w\vDash^+\phi$\\
$w:\phi;\mathfrak{f}$&$\mathfrak{M},w\vDash^-\phi$\\
$w:\phi;\overline{\mathfrak{t}}$&$\mathfrak{M},w\nvDash^+\phi$\\
$w:\phi;\overline{\mathfrak{f}}$&$\mathfrak{M},w\nvDash^-\phi$
\end{tabular}
\end{center}
\begin{convention}\label{conv:conjugatesandinverses}
We set
\begin{align*}
\overline{\overline{\mathfrak{t}}}&=\mathfrak{t}&\overline{\overline{\mathfrak{f}}}&=\mathfrak{f}&
\mathfrak{t}^\neg&=\mathfrak{f}&\mathfrak{f}^\neg&=\mathfrak{t}&
\overline{\mathfrak{t}}^\neg&=\overline{\mathfrak{f}}&\overline{\mathfrak{f}}^\neg&=\overline{\mathfrak{t}}
\end{align*}

For $\mathfrak{v}_1,\mathfrak{v}_2\in\mathsf{Val}$, we will write $w:\phi;\mathfrak{v_1};\mathfrak{v_2}$ as a~shorthand for $\{w:\phi;\mathfrak{v_1},w:\phi;\mathfrak{v_2}\}$.
\end{convention}

Let us now define the calculus formally.
\begin{definition}[$\mathbb{S}(\AFDE)$ --- analytic cut system for $\AFDE$]\label{def:AFDErules}
We define a~$\mathbb{S}(\AFDE)$-proof as a~downward branching tree whose nodes are labelled with sets containing labelled formulas and constructions of the form $w\mathsf{R}w'$. Each branch can be extended by one of the following rules (below, $w_{k_i}$'s are fresh in the branch, $i\neq j$).
\begin{spacing}{1}
\[\begin{array}{cccc}
\neg\mathfrak{t}:\dfrac{w:\neg\phi;\mathfrak{t}}{w:\phi;\mathfrak{f}}&\neg\mathfrak{f}:\dfrac{w:\neg\phi;\mathfrak{f}}{w:\phi;\mathfrak{t}}&\neg\overline{\mathfrak{t}}:\dfrac{w:\neg\phi;\overline{\mathfrak{t}}}{w:\phi;\overline{\mathfrak{f}}}&\neg\overline{\mathfrak{f}}:\dfrac{w:\neg\phi;\overline{\mathfrak{f}}}{w:\phi;\overline{\mathfrak{t}}}
\end{array}\]
\vspace{.5em}
\[\begin{array}{cccc}
\wedge\mathfrak{t}:\dfrac{w:\phi_1\wedge\phi_2;\mathfrak{t}}{\begin{matrix}w:\phi_1;\mathfrak{t}\\w:\phi_2;\mathfrak{t};\end{matrix}}&\wedge\mathfrak{f}:\dfrac{\begin{matrix}w:\phi_1\wedge\phi_2;\mathfrak{f}\\w:\phi_i;\overline{\mathfrak{f}}\end{matrix}}{w:\phi_j;\mathfrak{f}}&\wedge\overline{\mathfrak{t}}:\dfrac{\begin{matrix}w:\phi_1\wedge\phi_2;\overline{\mathfrak{t}}\\w:\phi_i;\mathfrak{t};\end{matrix}}{w:\phi_j;\overline{\mathfrak{t}}}&\wedge\overline{\mathfrak{f}}:\dfrac{w:\phi_1\wedge\phi_2;\overline{\mathfrak{f}}}{\begin{matrix}w:\phi_1;\overline{\mathfrak{f}}\\w:\phi_2;\overline{\mathfrak{f}}\end{matrix}}
\end{array}\]
\vspace{.5em}
\[\begin{array}{cccc}
\vee\mathfrak{t}:\dfrac{\begin{matrix}w:\phi_1\vee\phi_2;\mathfrak{t}\\w:\phi_i;\overline{\mathfrak{t}}\end{matrix}}{w:\phi_j;\mathfrak{t}}&\vee\mathfrak{f}:\dfrac{w:\phi_1\vee\phi_2;\mathfrak{f}}{\begin{matrix}w:\phi_1;\mathfrak{f}\\w:\phi_2;\mathfrak{f}\end{matrix}}&\vee\overline{\mathfrak{t}}:\dfrac{w:\phi_1\vee\phi_2;\overline{\mathfrak{t}}}{\begin{matrix}w:\phi_1;\overline{\mathfrak{t}}\\w:\phi_2;\overline{\mathfrak{t}}\end{matrix}}&\vee\overline{\mathfrak{f}}:\dfrac{\begin{matrix}w:\phi_1\vee\phi_2;\overline{\mathfrak{f}}\\w:\phi_i;\mathfrak{f}\end{matrix}}{w:\phi_j;\overline{\mathfrak{f}}}\\
\end{array}\]

\vspace{.3em}

\[\begin{array}{c}
\mathfrak{v}\overline{\mathfrak{v}}:\dfrac{}{w:\phi;\mathfrak{v}\mid w:\phi;\overline{\mathfrak{v}}}~\left(\parbox{15em}{$\phi$ is a~subformula of a~formula occurring on the branch; $w$ occurs on the branch}\right)
\end{array}\]

\vspace{.3em}

\[\begin{array}{ccc}
\blacktriangle_T\!:\!\dfrac{\begin{matrix}w_i\!:\!\blacktriangle\phi;\mathfrak{t};\overline{\mathfrak{f}}\\ w_i\mathsf{R}w_j\\w_j\!:\!\phi;\mathfrak{v}\end{matrix}}{w_j\!:\!\phi;\overline{\mathfrak{v}}^\neg}
&
\blacktriangle'_T\!:\!\dfrac{\begin{matrix}w_i\!:\!\blacktriangle\phi;\mathfrak{t};\overline{\mathfrak{f}}\\ w_i\mathsf{R}w_{j_1}\\w_i\mathsf{R}w_{j_2}\\w_{j_1}\!:\!\phi;\mathfrak{v};\overline{\mathfrak{v}}^\neg\end{matrix}}{w_{j_2}\!:\!\phi;\mathfrak{v};\overline{\mathfrak{v}}^\neg}
&
\blacktriangle_F\dfrac{w_i\!:\!\blacktriangle\phi;\mathfrak{f};\overline{\mathfrak{t}}}{\dfrac{\begin{matrix}w_i\mathsf{R}w_{k_1}\\w_i\mathsf{R}w_{k_2}\end{matrix}}{\left.\begin{matrix}w_{k_1}\!:\!\phi;\mathfrak{t}\\w_{k_2}\!:\!\phi;\overline{\mathfrak{t}}\end{matrix}\right|\begin{matrix}w_{k_1}\!:\!\phi;\mathfrak{f}\\w_{k_2}\!:\!\phi;\overline{\mathfrak{f}}\end{matrix}}}
\end{array}\]
\[\begin{array}{cccc}
\blacktriangle{}_{B}\!:\!\dfrac{\begin{matrix}w_i\!:\!\blacktriangle\phi;\mathfrak{t};\mathfrak{f}\\w_i\mathsf{R}w_j\end{matrix}}{\begin{matrix}w_j\!:\!\phi;\mathfrak{t};\mathfrak{f}\end{matrix}}
&
\blacktriangle{}^+_{B}\!:\!\dfrac{w_i\!:\!\blacktriangle\phi;\mathfrak{t};\mathfrak{f}}{\begin{matrix}w_i\mathsf{R}w_k\\w_k:\phi;\mathfrak{t};\mathfrak{f}\end{matrix}}
&
\blacktriangle{}_{N}\!:\!\dfrac{\begin{matrix}w_i\!:\!\blacktriangle\phi;\overline{\mathfrak{t}};\overline{\mathfrak{f}}\\w_i\mathsf{R}w_j\end{matrix}}{\begin{matrix}w_j:\phi;\overline{\mathfrak{t}};\overline{\mathfrak{f}}\end{matrix}}
&
\blacktriangle{}^+_{N}\!:\!\dfrac{w_i\!:\!\blacktriangle\phi;\overline{\mathfrak{t}};\overline{\mathfrak{f}}}{\begin{matrix}w_i\mathsf{R}w_k\\w_k:\phi;\overline{\mathfrak{t}};\overline{\mathfrak{f}}\end{matrix}}
\end{array}\]

\vspace{2em}
\end{spacing}

We say that a~branch $\mathcal{B}$ is closed iff the following condition is met. Otherwise, $\mathcal{B}$ is open.
\begin{itemize}
\item\label{closure1} $w_i:\phi;\mathfrak{v};\overline{\mathfrak{v}}\in\mathcal{B}$ for some $\phi\in\Ltriangle$, $w_i\in\mathsf{Lab}$, and $\mathfrak{v},\overline{\mathfrak{v}}\in\mathsf{Val}$.
\end{itemize}

An open branch $\mathcal{B}$ is \emph{complete} iff the following condition is met.
\begin{itemize}
\item If all premises of a~rule occur on the branch, then the conclusion\footnote{In the case of $\blacktriangle_F$ rule, at least one of two its conclusions should appear on the branch.} occurs on the branch as well.
\end{itemize}

A tree is closed iff every branch is closed.

Finally, we say that $\phi\vdash\chi$ is proved in $\mathbb{S}(\AFDE)$ iff there is a~closed tree whose root is $\{w\!:\!\phi;\mathfrak{t},~w\!:\!\chi;\overline{\mathfrak{t}}\}$.
\end{definition}
\begin{remark}
Let us clarify how the modal rules work. As one can see, they correspond to each Belnapian value of $\blacktriangle\phi$.
\begin{itemize}
\item $\blacktriangle_T$ guarantees that if $w_i\!:\!\blacktriangle\phi;\mathfrak{t};\overline{\mathfrak{f}}\in\mathcal{B}$, then $w_j:\phi;\mathfrak{t};\overline{\mathfrak{f}}\in\mathcal{B}$ for every $w_j$ s.t.\ $w_i\mathsf{R}w_j\in\mathcal{B}$ or $w_j:\phi;\mathfrak{f};\overline{\mathfrak{t}}\in\mathcal{B}$ for every $w_j$ s.t.\ $w_i\mathsf{R}w_j\in\mathcal{B}$. $\blacktriangle'_T$ ensures the closure of a~branch containing $w_i\!:\!\blacktriangle\phi;\mathfrak{t};\overline{\mathfrak{f}}$, $w_i\mathsf{R}w_{j_1}$, $w_i\mathsf{R}w_{j_2}$, $w_{j_1}\!:\!\phi;\mathfrak{t};\overline{\mathfrak{f}}$, and $w_{j_1}\!:\!\phi;\mathfrak{f};\overline{\mathfrak{t}}$.
\item $\blacktriangle_F$ adds two new accessible states to a~branch containing $w_i:\blacktriangle\phi;\overline{\mathfrak{t}};\mathfrak{f}$ and then splits the branch in two and gives $\phi$ different values in these added states.
\item $\blacktriangle^+_B$ adds a~new accessible state $w_k$ to a~branch containing $w_i\!:\!\blacktriangle\phi;\mathfrak{t};\mathfrak{f}$, s.t.\ $w_k\!:\!\phi;\mathfrak{t};\mathfrak{f}$. $\blacktriangle_B$ ensures that $w_j\!:\!\phi;\mathfrak{t};\mathfrak{f}$ in every accessible $w_j$.
\item Finally, $\blacktriangle^+_N$ and $\blacktriangle_N$ work dually to $\blacktriangle^+_B$ and $\blacktriangle_B$.
\end{itemize}
\end{remark}

We end the section with two proof trees: a~successful proof of $\blacktriangle p\vdash\blacktriangle\neg p$, and a~failed proof of $\blacktriangle(q\vee\neg q)$. For the latter, we show how to extract a~countermodel from a~complete open branch.
\begin{example}[Proofs]
A proof of $\blacktriangle p\vdash\blacktriangle\neg p$ is given in fig.~\ref{fig:goodproof}. A~failed proof can be seen in fig.~\ref{fig:badproof}. For the sake of brevity, we will not apply the $\mathfrak{v}\overline{\mathfrak{v}}$ rule to $q\vee\neg q$ at $w_0$ as it is clear that these applications will not make any open branch closed.
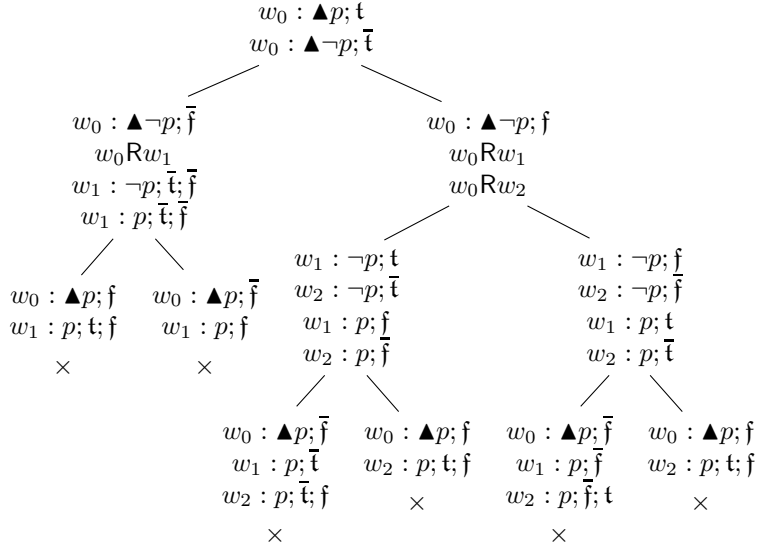
\begin{figure}
\centering
\begin{forest}
smullyan tableaux
[w_0:\blacktriangle p;\mathfrak{t}
[w_0:\blacktriangle\neg p;\overline{\mathfrak{t}}
[w_0:\blacktriangle\neg p;\overline{\mathfrak{f}}
[w_0\mathsf{R}w_1[w_1:\neg p;\overline{\mathfrak{t}};\overline{\mathfrak{f}}[w_1:p;\overline{\mathfrak{t}};\overline{\mathfrak{f}}
[w_0:\blacktriangle p;\mathfrak{f}[w_1:p;\mathfrak{t};\mathfrak{f},closed]][w_0:\blacktriangle p;\overline{\mathfrak{f}}[w_1:p;\mathfrak{f},closed]]
]]]
]
[w_0:\blacktriangle\neg p;\mathfrak{f}[w_0\mathsf{R}w_1[w_0\mathsf{R}w_2
[w_1:\neg p;\mathfrak{t}[w_2:\neg p;\overline{\mathfrak{t}}[w_1:p;\mathfrak{f}[w_2:p;\overline{\mathfrak{f}}[w_0:\blacktriangle p;\overline{\mathfrak{f}}[w_1:p;\overline{\mathfrak{t}}[w_2:p;\overline{\mathfrak{t}};\mathfrak{f},closed]]][w_0:\blacktriangle p;\mathfrak{f}[w_2:p;\mathfrak{t};\mathfrak{f},closed]]
]]]]
[w_1:\neg p;\mathfrak{f}[w_2:\neg p;\overline{\mathfrak{f}}[w_1:p;\mathfrak{t}[w_2:p;\overline{\mathfrak{t}}[w_0:\blacktriangle p;\overline{\mathfrak{f}}[w_1:p;\overline{\mathfrak{f}}[w_2:p;\overline{\mathfrak{f}};\mathfrak{t},closed]]][w_0:\blacktriangle p;\mathfrak{f}[w_2:p;\mathfrak{t};\mathfrak{f},closed]]
]]]]]]]]]]]]]
\end{forest}
\caption{A successful proof: all branches are closed.}
\label{fig:goodproof}
\end{figure}
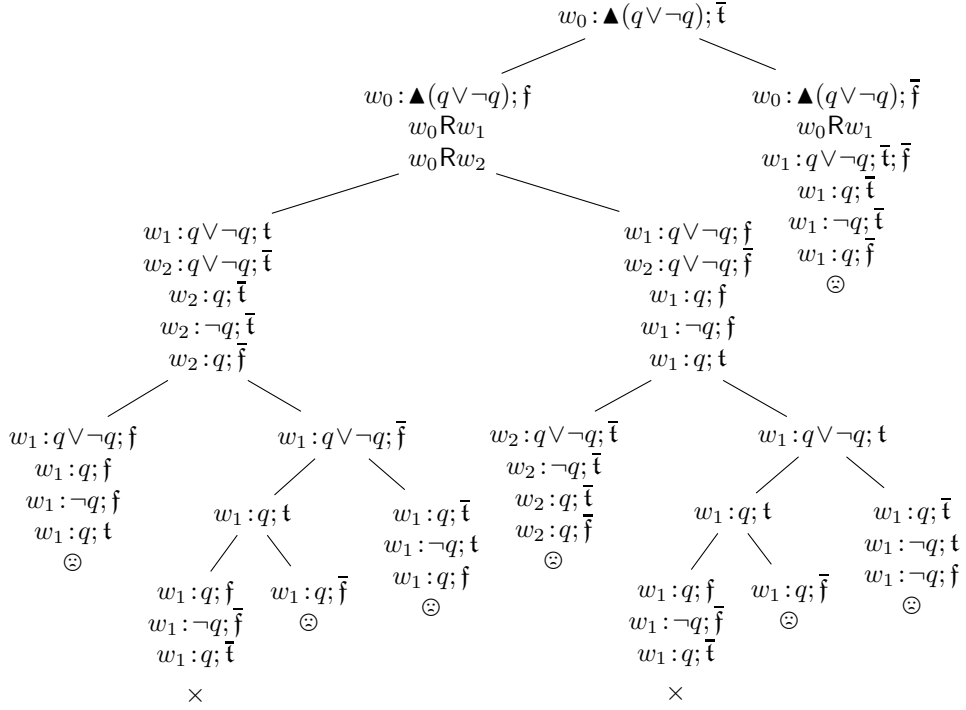
\begin{figure}
\centering
\begin{forest}
smullyan tableaux
[w_0\!:\!\blacktriangle(q\!\vee\!\neg q);\overline{\mathfrak{t}}
[w_0\!:\!\blacktriangle(q\!\vee\!\neg q);\mathfrak{f}[w_0\mathsf{R}w_1[w_0\mathsf{R}w_2
[w_1\!:\!q\!\vee\!\neg q;\mathfrak{t}[w_2\!:\!q\!\vee\!\neg q;\overline{\mathfrak{t}}[w_2\!:\!q;\overline{\mathfrak{t}}[w_2\!:\!\neg q;\overline{\mathfrak{t}}[w_2\!:\!q;\overline{\mathfrak{f}}
[w_1\!:\!q\!\vee\!\neg q;\mathfrak{f}[w_1\!:\!q;\mathfrak{f}[w_1\!:\!\neg q;\mathfrak{f}[w_1\!:\!q;\mathfrak{t}[\frownie]]]]][w_1\!:\!q\!\vee\!\neg q;\overline{\mathfrak{f}}
[w_1\!:\!q;\mathfrak{t}
[w_1\!:\!q;\mathfrak{f}[w_1\!:\!\neg q;\overline{\mathfrak{f}}[w_1\!:\!q;\overline{\mathfrak{t}},closed]]][w_1\!:\!q;\overline{\mathfrak{f}}[\frownie]]]
[w_1\!:\!q;\overline{\mathfrak{t}}[w_1\!:\!\neg q;\mathfrak{t}[w_1\!:\!q;\mathfrak{f}[\frownie]]]]
]
]]]]]
[w_1\!:\!q\!\vee\!\neg q;\mathfrak{f}[w_2\!:\!q\!\vee\!\neg q;\overline{\mathfrak{f}}[w_1\!:\!q;\mathfrak{f}[w_1\!:\!\neg q;\mathfrak{f}[w_1\!:\!q;\mathfrak{t}
[w_2\!:\!q\!\vee\!\neg q;\overline{\mathfrak{t}}[w_2\!:\!\neg q;\overline{\mathfrak{t}}[w_2\!:\!q;\overline{\mathfrak{t}}[w_2\!:\!q;\overline{\mathfrak{f}}[\frownie]]]]][w_1\!:\!q\!\vee\!\neg q;\mathfrak{t}
[w_1\!:\!q;\mathfrak{t}
[w_1\!:\!q;\mathfrak{f}[w_1\!:\!\neg q;\overline{\mathfrak{f}}[w_1\!:\!q;\overline{\mathfrak{t}},closed]]][w_1\!:\!q;\overline{\mathfrak{f}}[\frownie]]
]
[w_1\!:\!q;\overline{\mathfrak{t}}[w_1\!:\!\neg q;\mathfrak{t}[w_1\!:\!\neg q;\mathfrak{f}[\frownie]]]]
]
]]]]]
]]][w_0\!:\!\blacktriangle(q\!\vee\!\neg q);\overline{\mathfrak{f}}[w_0\mathsf{R}w_1[w_1\!:\!q\!\vee\!\neg q;\overline{\mathfrak{t}};\overline{\mathfrak{f}}[w_1\!:\!q;\overline{\mathfrak{t}}[w_1\!:\!\neg q;\overline{\mathfrak{t}}[w_1\!:\!q;\overline{\mathfrak{f}}[\frownie]]]]]]]
]
]
\end{forest}
\caption{A failed proof: complete open branches are denoted with $\frownie$.}
\label{fig:badproof}
\end{figure}

Let us now extract the countermodel of $\blacktriangle(q\vee\neg q)$ that corresponds to the leftmost complete open branch (fig.~\ref{fig:badproof}). As one can see, we need three states: $w_0$, $w_1$, and $w_2$, s.t.\ $R(w_0)=\{w_1,w_2\}$. It remains to use the entries of the branch to set the valuations. Note that the branch does not give values to $q$ at $w_0$ because we did not apply $\mathfrak{v}\overline{\mathfrak{v}}$ to $q\vee\neg q$, so we set them in an arbitrary manner. The model can be seen in fig.~\ref{fig:badproofexample}.
\begin{figure}
\centering
\begin{tikzpicture}[modal,node distance=0.5cm,world/.append style={minimum
size=1cm}]
\node[world] (w0) [label=below:{$w_0$}] {$q^\pm$};
\node[world] (w1) [left=of w0] [label=below:{$w_1$}] {$q^\pm$};
\node[world] (w2) [right=of w0] [label=below:{$w_2$}] {$\xcancel{q}$};
\path[->] (w0) edge (w1);
\path[->] (w0) edge (w2);
\end{tikzpicture}
\caption{The model corresponding to the leftmost open branch of fig.~\ref{fig:badproof}.}
\label{fig:badproofexample}
\end{figure}
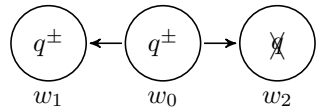
\end{example}
\subsection{Soundness and completeness}
We are now ready to prove the soundness and completeness results for our systems. We adapt the approach from~\cite{DAgostino1990}.
\begin{definition}[Branch realisation]\label{branchrealisation}
We say that $\mathfrak{M}=\langle W,R,v^+,v^-\rangle$ with $W=\{w:w\text{ occurs on }\mathcal{B}\}$, $R=\{\langle w_i,w_j\rangle:w_i\mathsf{R}w_j\in\mathcal{B}\}$, and $w\in v^+(p)$ ($w\in v^-(p)$) iff $w\!:\!p;\mathfrak{t}\in\mathcal{B}$ ($w\!:\!p;\mathfrak{f}\in\mathcal{B}$) realises a~branch $\mathcal{B}$ of a~tree iff the following conditions are met.
\begin{enumerate}
\item If $w\!:\!\phi;\mathfrak{t}\in\mathcal{B}$ ($w\!:\!\phi;\mathfrak{f}\in\mathcal{B}$), then $\mathfrak{M},w\vDash^+\phi$ ($\mathfrak{M},w\vDash^-\phi$, respectively).
\item If $w\!:\!\phi;\overline{\mathfrak{t}}\in\mathcal{B}$ ($w\!:\!\phi;\overline{\mathfrak{f}}\in\mathcal{B}$), then $\mathfrak{M},w\nvDash^+\phi$ ($\mathfrak{M},w\nvDash^-\phi$, respectively).
\end{enumerate}
\end{definition}
\begin{theorem}[Soundness of $\mathbb{S}(\AFDE)$]\label{AFDEsoundness}
If there is a~proof of $\phi\vdash\chi$ in $\mathbb{S}(\AFDE)$, then $\phi\vdash\chi$ is valid.
\end{theorem}
\begin{proof}
It is easy to check that the rules are sound in the sense that if a~branch is realised by a~model, then its extension by any rule is realised too. On the other hand, a~closed branch clearly cannot be realised. Thus, if the tree is closed, then the initial labelled formulas are not realisable. But in order to prove $\phi\vdash\chi$, we start a~tree with $\{\phi\!:\!\mathfrak{t};w_0,\chi\!:\!\overline{\mathfrak{t}};w_0\}$. Hence, if this set cannot be realised, the sequent is valid.
\end{proof}
\begin{theorem}[Completeness of $\mathbb{S}(\AFDE)$]\label{AFDEcompleteness}
Any $\AFDE$-valid sequent $\phi\vdash\chi$ is provable, i.e., there is a~closed tree whose initial node is $\{w_0:\phi;\mathfrak{t},w_0:\chi;\overline{\mathfrak{t}}\}$.
\end{theorem}
\begin{proof}
We proceed by contraposition. We need to show that every complete open branch is realisable. We now show by induction on $\phi$ that $w:\phi;\mathfrak{t}$ iff $\mathfrak{M},w\vDash^+\phi$ and $w:\phi;\mathfrak{f}$ iff $\mathfrak{M},w\vDash^-\phi$ with $\mathfrak{M}$ as in Definition~\ref{branchrealisation}. The basis case of $\phi=p$ holds by the construction of $\mathfrak{M}$. The propositional cases are straightforward. Thus, we are going to consider only the most instructive case of $\phi=\blacktriangle\phi'$.

Let $w_i:\blacktriangle\phi';\mathfrak{t}\in\mathcal{B}$. Since $\mathcal{B}$ is complete, either (1) $w_i:\blacktriangle\phi';\mathfrak{f}\in\mathcal{B}$ or (2) $w_i:\blacktriangle\phi';\overline{\mathfrak{f}}\in\mathcal{B}$ ($\blacktriangle\phi'$ is its own subformula and $w_i$ is on the branch, so we can apply $\mathfrak{v}\overline{\mathfrak{v}}$ rule).

Consider (1). By rules $\blacktriangle{}B$ and $\blacktriangle{}B^+$, there is at least one state $w_j$ s.t.\ $w_i\mathsf{R}w_j\in\mathcal{B}$ and for all such $w_j$'s we have $w_j:\phi';\mathfrak{t};\mathfrak{f}$. Hence, by the induction hypothesis, there are some states $w_j\in R(w_i)$ and in all of them, $\mathfrak{M},w_j\vDash^+\phi'$ and $\mathfrak{M},w_j\vDash^-\phi'$. But then $\mathfrak{M},w_i\vDash^+\blacktriangle\phi'$ and $\mathfrak{M},w_i\vDash^-\blacktriangle\phi'$.

Consider (2). Then $w_i:\blacktriangle\phi';\mathfrak{t};\overline{\mathfrak{f}}\in\mathcal{B}$. If there is no $w_j$ s.t.\ $w_i\mathsf{R}w_j\in\mathcal{B}$, then $R(w_i)=\varnothing$, whence $\mathfrak{M},w_i\vDash^+\blacktriangle\phi'$ and $\mathfrak{M},w_i\nvDash^-\blacktriangle\phi'$, as required. Otherwise, since $\mathcal{B}$ is open, by rules $\blacktriangle'_T$ and $\blacktriangle_T$, we have two options. (i) $w_j\!:\!\phi';\mathfrak{t};\overline{\mathfrak{f}}\in\mathcal{B}$ for every $w_j$ s.t.\ $w_i\mathsf{R}w_j\in\mathcal{B}$ or (ii) $w_j\!:\!\phi';\mathfrak{f};\overline{\mathfrak{t}}\in\mathcal{B}$ for every $w_j$ s.t.\ $w_i\mathsf{R}w_j\in\mathcal{B}$.

In the first case, we obtain that $\mathfrak{M},w_j\vDash^+\phi$ and $\mathfrak{M},w_j\nvDash^-\phi$ for every $w_j\in R(w_i)$, by the induction hypothesis. Then, $\mathfrak{M},w_i\vDash^+\blacktriangle\phi$ and $\mathfrak{M},w_i\nvDash^-\blacktriangle\phi$, as required. In the second case, we obtain that $\mathfrak{M},w_j\nvDash^+\phi$ and $\mathfrak{M},w_j\vDash^-\phi$ for every $w_j\in R(w_i)$, by the induction hypothesis. Then, again, $\mathfrak{M},w_i\vDash^+\blacktriangle\phi$ and $\mathfrak{M},w_i\nvDash^-\blacktriangle\phi$, as required.

To show the converse, assume that $w_i\blacktriangle\phi';\mathfrak{t}\notin\mathcal{B}$. By completeness of $\mathcal{B}$ we have that $w_i:\blacktriangle\phi';\overline{\mathfrak{t}}\in\mathcal{B}$ and exactly one of the following options hold: either (i)~$w_i:\blacktriangle\phi';\mathfrak{f}\in\mathcal{B}$ or (ii) $w_i:\blacktriangle\phi';\overline{\mathfrak{f}}\in\mathcal{B}$.

In the first case, using $\blacktriangle_F$, we add new states $w_{k_1}$ and $w_{k_2}$ s.t.\ $w_i\mathsf{R}w_{k_1}\in\mathcal{B}$ and $w_i\mathsf{R}w_{k_2}\in\mathcal{B}$. Moreover, $\blacktriangle_F$ gives us that
\begin{enumerate}
\item[(a)] $w_{k_1}\!:\!\phi;\mathfrak{t}\in\mathcal{B}$  and $w_{k_2}\!:\!\phi;\overline{\mathfrak{t}}\in\mathcal{B}$, or
\item[(b)] $w_{k_1}\!:\!\phi;\mathfrak{f}\in\mathcal{B}$  and $w_{k_2}\!:\!\phi;\overline{\mathfrak{f}}\in\mathcal{B}$.
\end{enumerate}
By the induction hypothesis, we have that either $\mathfrak{M},w_{k_1}\vDash^+\phi'$ and $\mathfrak{M},w_{k_2}\nvDash^+\phi'$ or $\mathfrak{M},w_{k_1}\vDash^-\phi'$ and $\mathfrak{M},w_{k_2}\nvDash^-\phi'$. But these are precisely conditions~\eqref{f1conditionI} and~\eqref{f2conditionI} from Definition~\ref{def:AFDEsemantics}. Thus, $\mathfrak{M},w_i\vDash^-\blacktriangle\phi'$. But note that condition~\eqref{t1conditionI} is failed, whence $\mathfrak{M},w_i\nvDash^+\blacktriangle\phi'$.

In the second case, $\blacktriangle_N$ and $\blacktriangle^+_N$ guarantee us that there is some $w'$ s.t.\ $w_i\mathsf{R}w'\in\mathcal{B}$ and that $w_k:\phi';\overline{\mathfrak{t}};\overline{\mathfrak{f}}\in\mathcal{B}$ for all $w_k$ s.t.\ $w_i\mathsf{R}w_k\in\mathcal{B}$. By the induction hypothesis, it entails that for any state $w_k\in R(w_i)$ $\mathfrak{M},w_k\nvDash^+\phi'$ and $\mathfrak{M},w_k\nvDash^-\phi'$. Which means that conditions~\eqref{t1conditionI}, \eqref{fconditionS}, \eqref{f1conditionI}, and \eqref{f2conditionI} are failed and thus $\mathfrak{M},w_i\nvDash^+\blacktriangle\phi'$ and $\mathfrak{M},w_i\nvDash^-\blacktriangle\phi'$ as desired.

The case of $w_i:\blacktriangle\phi';\mathfrak{t}\in\mathcal{B}$ can be tackled in the same manner.
\end{proof}

We end the section with two corollaries from completeness. First of all, using Theorem~\ref{theorem:AFDEcontraposition}, we obtain the following statement.
\begin{corollary}\label{KEtheorem:AFDEcontraposition}
The $\mathbb{S}(\AFDE)$ tree for $\{\phi\!:\!\mathfrak{t};w,\chi\!:\!\overline{\mathfrak{t}};w\}$ is closed iff the tree for $\{\phi\!:\!\overline{\mathfrak{f}};w,\chi\!:\!\mathfrak{f};w\}$ is closed.
\end{corollary}
Secondly, completeness gives rise to the subformula property.
\begin{corollary}\label{subformulaproperty}
If there is a~$\mathbb{S}(\AFDE)$ proof of $\phi\vdash\chi$, then all labelled formulas appearing in it are subformulas of either $\phi$ or $\chi$.
\end{corollary}
\section{Expressivity of $\AFDE$}\label{sec:expressivity}
Recall that the classical non-contingency operator $\triangle$ can be defined in terms of the classical necessity operator $\Box$ on any class of frames as follows: $\Box p\vee\Box\neg p\equiv\triangle p$. On the other hand, $\Box$ cannot be defined in terms of $\triangle$ unless the underlying frame is reflexive. It means (cf., e.g.~\cite[Propositions~3.2--3.4]{FanWangvanDitmarsch2015}) that (unless the frame is reflexive) $\triangle$ is strictly less expressive than $\Box$. It is thus instructive to compare expressivity of $\blacktriangle$ with that of $\Box$ as it is given in a~modal expansion of $\mathbf{FDE}$ (cf., e.g.~\cite[\S11a.4]{Priest2008FromIftoIs} and~\cite[\S2]{OdintsovWansing2017}) dubbed $\mathbf{K_{FDE}}$\footnote{Note, however, that while it is the most straightforward definition of $\Box$ in Belnap --- Dunn Logic, there are several alternative definitions given, e.g., in~\cite{Drobyshevich2020}. Moreover, to the best of our knowledge, there is no general notion of a~‘modal logic over $\mathbf{FDE}$’, in contrast to ‘classical modal logic’.}.
\begin{equation}
\label{tconditionBox}
\tag{$t\Box$}
\mathfrak{M},w\vDash^+\Box\phi\Leftrightarrow\forall w':\text{if }R(w,w'),\text{ then }\mathfrak{M},w'\vDash^+\phi
\end{equation}
\begin{equation}
\label{fconditionBox}
\tag{$f\Box$}
\mathfrak{M},w\vDash^-\Box\phi\Leftrightarrow\exists w':R(w,w')\text{ and }\mathfrak{M},w'\vDash^-\phi
\end{equation}
\begin{convention}
We will further denote the language over $\{\neg,\wedge,\vee,\Box\}$ with $\Lbox$. $\lozenge\phi$~is a~shorthand for $\neg\Box\neg\phi$ and $\blacktriangledown\phi$ is a~shorthand for $\neg\blacktriangle\phi$.
\end{convention}

\begin{definition}\label{def:formuladefinability}
Let $\mathcal{L}_1$ and $\mathcal{L}_2$ be two languages and let $\mathbb{K}$ be a~class of frames. We say that $\phi\in\mathcal{L}_1$ \emph{defines} $\chi\in\mathcal{L}_2$ \emph{in $\mathbb{K}$} iff for any $\mathfrak{F}\in\mathbb{K}$ and for any pointed model $\langle\mathfrak{M},w\rangle$ on $\mathfrak{F}$, it holds that
\begin{align*}
\mathfrak{M},w\vDash^+\phi&\text{ iff }\mathfrak{M},w\vDash^+\chi&
\mathfrak{M},w\vDash^-\phi&\text{ iff }\mathfrak{M},w\vDash^-\chi
\end{align*}
\end{definition}

Note first of all, that in contrast to the classical case $\blacktriangle p$ is not definable as $\Box p\vee\Box\neg p$. The model in fig.~\ref{fig:trueandfalse} serves as a~counterexample: $\mathfrak{M},w_0\nvDash^+\blacktriangle p$ while $\mathfrak{M},w_0\vDash^+\Box p\vee\Box\neg p$.

This is reasonable to expect since $\Box p\vee\Box\neg p$ is interpreted as ‘p is true in all accessible states or false in all accessible states’ but $\blacktriangle p$ means ‘the value of $p$ is the same in all accessible states’ which is a~stronger condition. Indeed, one can see from \eqref{tconditionBox} and \eqref{fconditionBox} that $\blacktriangle p\vdash\Box p\vee\Box\neg p$ is valid.

The following theorem shows that no $\Lbox$ formula can define $\blacktriangle p$ at all.
\begin{theorem}\label{theorem:trianglenotdefinable}
There is no formula $\phi\in\Lbox$ that defines $\blacktriangle p$ on the classes of all frames, all reflexive frames, all transitive frames, all symmetric frames, and all Euclidean frames.
\end{theorem}
\begin{proof}
Consider the models in fig.~\ref{fig:S4counterexample}.
\begin{figure}
\centering
\begin{tikzpicture}[modal,node distance=0.5cm,world/.append style={minimum
size=1cm}]
\node[world] (w0) [label=below:{$w_0$}] {$p^+$};
\node[] [left=of w0] {$\mathfrak{M}$:};
\path[->] (w0) edge[reflexive] (w0);
\end{tikzpicture}
\hfil
\begin{tikzpicture}[modal,node distance=0.5cm,world/.append style={minimum
size=1cm}]
\node[world] (w0) [label=below:{$w'_0$}] {$p^+$};
\node[world] (w1) [right=of w0] [label=below:{$w'_1$}] {$\xcancel{p}$};
\node[] [left=of w] {$\mathfrak{M}'$:};
\path[->] (w0) edge[reflexive] (w0);
\path[->] (w1) edge[reflexive] (w1);
\path[<->] (w0) edge (w1);
\end{tikzpicture}
\caption{All variables in both models have the same values exemplified by $p$.}
\label{fig:S4counterexample}
\end{figure}
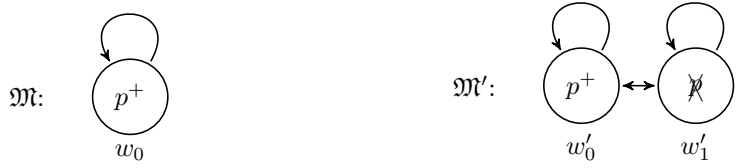
Clearly, $\mathfrak{M},w_0\vDash^+\blacktriangle p$ and $\mathfrak{M},w_0\nvDash^-\blacktriangle p$ but $\mathfrak{M}',w'_0\nvDash^+\blacktriangle p$ and $\mathfrak{M}',w'_0\vDash^-\blacktriangle p$. It now suffices to show for any $\phi\in\Lbox$ that
\begin{enumerate}
\item if $\mathfrak{M}',w'_0\nvDash^+\phi$ and $\mathfrak{M}',w'_0\vDash^-\phi$, then $\mathfrak{M},w_0\nvDash^+\phi$ and $\mathfrak{M},w_0\vDash^-\phi$; and
\item if $\mathfrak{M}',w'_0\vDash^+\phi$ and $\mathfrak{M}',w'_0\nvDash^-\phi$, then $\mathfrak{M},w_0\vDash^+\phi$ and $\mathfrak{M},w_0\nvDash^-\phi$.
\end{enumerate}

We proceed by induction on $\phi$. The basis case of variables is trivial since their valuations at $w_0$ and $w'_0$ are the same. We only show the most instructive cases of $\phi=\psi\wedge\psi'$ and $\phi=\lozenge\psi$.

\fbox{$\phi=\psi\wedge\psi'$}

For (1), if $\mathfrak{M}',w'_0\nvDash^+\psi\wedge\psi'$ and $\mathfrak{M}',w'_0\vDash^-\psi\wedge\psi'$, then
\begin{enumerate}
\item[(a)] $\mathfrak{M}',w'_0\nvDash^+\psi$ and $\mathfrak{M}',w'_0\vDash^-\psi$, or
\item[(b)] $\mathfrak{M}',w'_0\nvDash^+\psi'$ and $\mathfrak{M}',w'_0\vDash^-\psi'$, or
\item[(c)] w.l.o.g. $\mathfrak{M}',w'_0\vDash^+\psi$ and $\mathfrak{M}',w'_0\vDash^-\psi$ but $\mathfrak{M}',w'_0\nvDash^+\psi'$ and $\mathfrak{M}',w'_0\nvDash^-\psi'$.
\end{enumerate}

Cases (a) and (b) hold by the induction hypothesis. For (c), one can show by induction that there is no $\phi\in\Lbox$ s.t.\ $\mathfrak{M}',w'_0\vDash^+\phi$ and $\mathfrak{M}',w'_0\vDash^-\phi$. The basis case of propositional variables holds by the construction of the model, and the cases of propositional connectives can be shown by a~straightforward application of the induction hypothesis. Finally, one can see that there is no $\chi\in\Lbox$ s.t.\ $\mathfrak{M}',w'_1\vDash^+\chi$ and $\mathfrak{M}',w'_1\vDash^-\chi$. Thus, it holds that there is no $\phi=\Box\chi$, nor $\phi=\lozenge\chi$ s.t.\ $\mathfrak{M}',w'_0\vDash^+\phi$ and $\mathfrak{M}',w'_0\vDash^-\phi$.

For (2), recall that if $\mathfrak{M}',w'_0\vDash^+\psi\wedge\psi'$ and $\mathfrak{M}',w'_0\nvDash^-\psi\wedge\psi'$, then $\mathfrak{M}',w'_0\vDash^+\psi$ and $\mathfrak{M}',w'_0\nvDash^-\psi$ as well as $\mathfrak{M}',w'_0\vDash^+\psi'$ and $\mathfrak{M}',w'_0\nvDash^-\psi'$. Hence, by the induction hypothesis, $\mathfrak{M},w_0\vDash^+\psi$ and $\mathfrak{M},w_0\nvDash^-\psi$ as well as $\mathfrak{M},w_0\vDash^+\psi'$ and $\mathfrak{M},w_0\nvDash^-\psi'$. Thus, $\mathfrak{M}',w'_0\vDash^+\psi\wedge\psi'$ and $\mathfrak{M}',w'_0\nvDash^-\psi\wedge\psi'$, as required.

\fbox{$\phi=\lozenge\psi$}

Observe first, that since $\forall w',w''\!\in\!\mathfrak{M}':w'Rw''$, it holds that
\begin{itemize}
\item $\mathfrak{M}',w'\vDash^+\lozenge\psi$ iff $\mathfrak{M}',w''\vDash^+\lozenge\psi$, and
\item $\mathfrak{M}',w'\vDash^-\lozenge\psi$ iff $\mathfrak{M}',w''\vDash^-\lozenge\psi$
\end{itemize}
for any $w',w''\in\mathfrak{M}'$ and any $\lozenge\psi\in\Lbox$.

Now, consider (1). If $\mathfrak{M}',w'_0\nvDash^+\lozenge\psi$ and $\mathfrak{M}',w'_0\vDash^-\lozenge\psi$, then $\mathfrak{M}',w'\nvDash^+\psi$ and $\mathfrak{M}',w'\vDash^-\psi$ for any $w'\in\{w'_0,w'_1\}$. But then, $\mathfrak{M},w_0\nvDash^+\psi$ and $\mathfrak{M},w_0\vDash^-\psi$ by the induction hypothesis. Hence, $\mathfrak{M},w_0\nvDash^+\lozenge\psi$ and $\mathfrak{M},w_0\vDash^-\lozenge\psi$, as required.

For (2), let $\mathfrak{M}',w'_0\vDash^+\lozenge\psi$ and $\mathfrak{M}',w'_0\nvDash^-\lozenge\psi$. We have four options:
\begin{enumerate}
\item[(a)] $\mathfrak{M}',w'_0\vDash^+\psi$ and $\mathfrak{M}',w'_0\nvDash^-\psi$, or
\item[(b)] $\mathfrak{M}',w'_1\vDash^+\psi$ and $\mathfrak{M}',w'_1\nvDash^-\psi$, or
\item[(c)] $\mathfrak{M}',w'_1\vDash^+\psi$, $\mathfrak{M}',w'_1\vDash^-\psi$, $\mathfrak{M}',w'_0\nvDash^+\psi$ and $\mathfrak{M}',w'_0\nvDash^-\psi$, or
\item[(d)] $\mathfrak{M}',w'_0\vDash^+\psi$, $\mathfrak{M}',w'_0\vDash^-\psi$, $\mathfrak{M}',w'_1\nvDash^+\psi$ and $\mathfrak{M}',w'_1\nvDash^-\psi$.
\end{enumerate}

The (a) case holds by the induction hypothesis. Case (d) holds trivially since there is no $\phi\in\Lbox$ s.t.\ $\mathfrak{M}',w'_0\vDash^+\phi$ and $\mathfrak{M}',w'_0\vDash^-\phi$ as we have just shown. Case (c) holds trivially as well because a~similar inductive argument demonstrates that there is no $\phi\in\Lbox$ s.t.\ $\mathfrak{M}',w'_1\vDash^+\phi$ and $\mathfrak{M}',w'_1\vDash^-\phi$.

Finally, we reduce (b) to (a) by proving that for any $\psi\in\Lbox$, it holds that (i) if $\mathfrak{M}',w'_1\vDash^+\psi$ and $\mathfrak{M}',w'_1\nvDash^-\psi$, then $\mathfrak{M}',w'_0\vDash^+\psi$ and $\mathfrak{M}',w'_0\nvDash^-\psi$ as well, and that (ii) if $\mathfrak{M}',w'_1\nvDash^+\psi$ and $\mathfrak{M}',w'_1\vDash^-\psi$, then $\mathfrak{M}',w'_0\nvDash^+\psi$ and $\mathfrak{M}',w'_0\vDash^-\psi$. We proceed by induction on $\psi$.

The basis case of a~propositional variable holds trivially. The propositional cases are also straightforward. Now, if $\psi=\lozenge\chi$, we use the observation above to obtain that (i) and (ii) hold as well.

The result follows.
\end{proof}

Recall that in the classical non-contingency logic $p\wedge\triangle p$ \emph{defines} $\Box p$ over the class of reflexive frames. We can show that this is not the case for $\AFDE$ and $\mathbf{K}_{\mathbf{FDE}}$.
\begin{theorem}\label{theorem:boxnotdefinable}
There is no $\phi\in\Ltriangle$ that defines $\Box p$ over the classes of all, reflexive, symmetric, transitive, and Euclidean frames.
\end{theorem}
\begin{proof}
Consider the model in fig.~\ref{fig:S5counterexample}.
\begin{figure}
\centering
\begin{tikzpicture}[modal,node distance=0.5cm,world/.append style={minimum
size=1cm}]
\node[world] (w0) [label=below:{$w_0$}] {$p^+$};
\node[world] (w1) [right=of w0] [label=below:{$w_1$}] {$p^\pm$};
\node[] [left=of w0] {$\mathfrak{M}$:};
\path[->] (w0) edge[reflexive] (w0);
\path[->] (w1) edge[reflexive] (w1);
\path[<->] (w0) edge (w1);
\end{tikzpicture}
\caption{All variables have the same values exemplified by $p$.}
\label{fig:S5counterexample}
\end{figure}
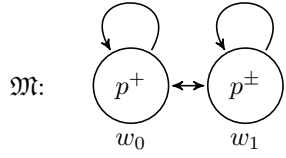
It is clear that $\mathfrak{M},w_0\vDash^+\Box p$ and $\mathfrak{M},w_0\vDash^-\Box p$. However, we can show that there is no $\phi\in\Ltriangle$ s.t.\ $\mathfrak{M},w_0\vDash^+\phi$ and $\mathfrak{M},w_0\vDash^-\phi$.

We proceed by induction on $\phi\in\Ltriangle$. Clearly, $\mathfrak{M},w_0\nvDash^-p$. Let us consider the remaining cases. We will reason for a~contradiction.

\fbox{$\phi=\neg\psi$}

Assume that $\mathfrak{M},w_0\vDash^+\neg\psi$ and $\mathfrak{M},w_0\vDash^-\neg\psi$. But then, $\mathfrak{M},w_0\vDash^+\psi$ and $\mathfrak{M},w_0\vDash^-\psi$. A~contradiction.

\fbox{$\phi=\psi\wedge\psi'$}

Assume that $\mathfrak{M},w_0\vDash^+\neg\psi$ and $\mathfrak{M},w_0\vDash^-\neg\psi$. But then, either $\mathfrak{M},w_0\vDash^+\psi$ and $\mathfrak{M},w_0\vDash^-\psi$, or $\mathfrak{M},w_0\vDash^+\psi'$ and $\mathfrak{M},w_0\vDash^-\psi'$. A~contradiction.

The case of $\phi=\psi\vee\psi'$ can be tackled in the same manner.

\fbox{$\phi=\blacktriangle\psi$}

Assume that $\mathfrak{M},w_0\vDash^+\blacktriangle\psi$ and $\mathfrak{M},w_0\vDash^-\blacktriangle\psi$. Then, we have $\mathfrak{M},w_0\vDash^+\psi$, $\mathfrak{M},w_0\vDash^-\psi$, $\mathfrak{M},w_1\vDash^+\psi$, and $\mathfrak{M},w_1\vDash^-\psi$. But by the induction hypothesis, there is no $\chi\in\Ltriangle$ s.t.\ it is less complex than $\blacktriangle\psi$, $\mathfrak{M},w_0\vDash^+\chi$, and $\mathfrak{M},w_0\vDash^-\chi$. A~contradiction.
\end{proof}

We would like to once again draw our readers' attention to Theorem~\ref{theorem:trianglenotdefinable}. This result shows us that in contrast to the classical case, where ‘$\phi$ is true in all accessible states or false in all accessible states’ is an equivalent reformulation of ‘the value of $\phi$ is the same in all accessible states’, these two readings in $\mathbf{FDE}$ are different. Not only that, none of these two interpretations can be reduced to the other. In other words, we cannot straightforwardly model the search of inconsistencies among the accessible sources of information in $\Lbox$ if we assume that a~source can give contradictory or incomplete information.

A natural question arises then --- whether it is reasonable to define $\blacktriangle\phi$ using that other reading, and whether $\blacktriangle\phi$ thus defined is going to have any interesting properties. We leave this for future research.
\section{Frame definability\label{sec:framedefinability}}
It is known that the classical language containing $\triangle$ is strictly weaker than the one with $\Box$ instead. As we have previously mentioned, all functional frames are equivalent w.r.t.\ classical formulas~\cite{Zolin1999,FanvanDitmarsch2015}. Thus, many natural first-order properties of frames such as seriality, transitivity, Euclidieanness, etc. are not definable using~$\triangle$.

In this section, we are going to show that several frame properties which are \emph{not definable} in the classical case are definable with $\AFDE$ sequents.

First of all, observe that in contrast to the classical case, we can distinguish between different partial-functional frames. Consider the frames in fig.~\ref{fig:functionaldistinguishable}.

\begin{figure}
\centering
\begin{tikzpicture}[modal,node distance=0.5cm,world/.append style={minimum
size=1cm}]
\node[world] (w0) [label=below:{$w_0$}] {};
\node[] [left=of w] {$\mathfrak{F}$:};
\end{tikzpicture}
\hfil
\begin{tikzpicture}[modal,node distance=0.5cm,world/.append style={minimum
size=1cm}]
\node[world] (w0) [label=below:{$w'_0$}] {};
\node[] [left=of w] {$\mathfrak{F}'$:};
\path[->] (w0) edge[reflexive] (w0);
\end{tikzpicture}
\caption{Two distinguishable partial-functional frames.}
\label{fig:functionaldistinguishable}
\end{figure}
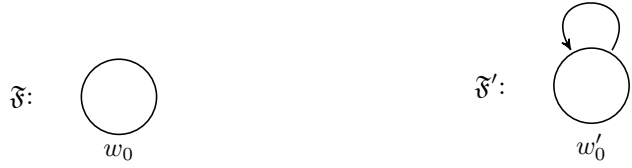

It is evident that $\blacktriangle p$ is valid on $\mathfrak{F}$ but not on $\mathfrak{F}'$. Indeed, since no state is accessible from $w_0$, $\mathfrak{M},w_0\vDash^+\blacktriangle p$ and $\mathfrak{M},w_0\nvDash^-\blacktriangle p$ for any $\mathfrak{M}$ on $\mathfrak{F}$. However, if we set $w'_0\notin v'^+(p)$ and $w'_0\notin v'^-(p)$ and let $\mathfrak{M}'=\langle\mathfrak{F}',v'^+,v'^-\rangle$, we will have that $\mathfrak{M}',w'_0\nvDash^+\blacktriangle p$ and $\mathfrak{M}',w'_0\nvDash^-\blacktriangle p$.
\begin{remark}
Observe that the possibility of distinguishing partial-functional frames is due to the fact that $\blacktriangle p$ is not necessarily true (and, a~fortiori, not necessarily true and not-false) at $w$ if $|R(w)|\leqslant1$. This provides another argument (albeit, a~technical one) in favour of our chosen semantics.

Indeed, had we chosen the support of truth condition of $\blacktriangle\phi$ at $w$ as ‘$\phi$ has the same Belnapian value in all successors of $w$’, then $\blacktriangle\phi$ would have been always true in all states of a~partial-functional frame for any $\phi$.
\end{remark}
This gives us an opportunity to provide some correspondence theory. The notion of frame definability will be a~slight modification of that in the classical logic.
\begin{definition}[Frame definability]
Let $\Sigma$ be a~set of ‘formula-formula’ sequents and $\mathbb{F}$ be a~class of frames. We say that $\Sigma$ \emph{defines} $\mathbb{F}$ iff for any frame $\mathfrak{F}$, $\mathfrak{F}\in\mathbb{F}$ iff all sequents from $\Sigma$ are valid on $\mathfrak{F}$. A~class of frames is definable in $\AFDE$ iff there is a~set of sequents that defines it.
\end{definition}

Below, we give notation for some classes of frames which are not definable with $\triangle$-formulas in the classical non-contingency logic. Of particular interest are frames $\langle W,R\rangle$ where $R$ is reflexive, an equivalence relation, or a~non-strict preorder (i.e., reflexive and transitive).
\begin{convention}[Classes of frames]\label{conv:classesofframes}
Let $\mathfrak{F}=\langle W,R\rangle$ be a~frame. We will use the following notation to designate different classes of frames.
\begin{center}
\begin{tabular}{c|c}
\textbf{notation}&\textbf{class of frames}\\\hline
$\mathbf{T}$&$R$ is reflexive\\
$\mathbf{S4}$&$R$ is transitive and reflexive\\
$\mathbf{S5}$&$R$ is an equivalence relation\\
$\mathbf{F}$&$R$ is partial-functional: $\forall x:\exists y,z(R(x,y)\&R(x,z))\Rightarrow y=z$\\
$\mathbf{Ver}$&$R=\varnothing$\\
$\mathbf{1}$&$R$ is coreflexive: $\forall x,y:R(x,y)\Rightarrow x=y$
\end{tabular}    
\end{center}
\end{convention}

The following statement is the main result of this section.
\begin{theorem}\label{AFDEcorrespondence}
All classes of frames from convention~\ref{conv:classesofframes} are definable in $\AFDE$.
\end{theorem}
\begin{proof}
\fbox{$\mathbf{T}$ frames}

We show that $\mathfrak{F}\in\mathbf{T}$ iff $\blacktriangle(p\vee\neg p)\vdash p\vee\neg p$ is valid on $\mathfrak{F}$.

Let $\mathfrak{F}$ be reflexive and $\mathfrak{M}$ be a~model on $\mathfrak{F}$ such that $\mathfrak{M},w\vDash^+\blacktriangle(p\vee\neg p)$. Then either $\mathfrak{M},w'\vDash^+p\vee\neg p$ and $\mathfrak{M},w'\nvDash^-p\vee\neg p$ in every $w'\in R(w)$ or $\mathfrak{M},w'\vDash^+p\vee\neg p$ and $\mathfrak{M},w'\vDash^-p\vee\neg p$ for it is never the case that $\mathfrak{M},w'\nvDash^+p\vee\neg p$ and $\mathfrak{M},w'\vDash^-p\vee\neg p$. In any case, since $w\in R(w)$, we have that $\mathfrak{M},w\vDash^+p\vee\neg p$, as required.

For the converse, let $\mathfrak{F}$ be non-reflexive, and further, $w\notin R(w)$. We define $v^+$ and $v^-$ as follows: $w\notin v^+(p)$, $w\notin v^-(p)$, and $w'\in v^+(p)$, $w'\notin v^-(p)$ for every $w'\in R(w)$. Now, it is clear that $\mathfrak{M},w\nvDash^+p\vee\neg p$. However, we have $\mathfrak{M},w\vDash^+\blacktriangle(p\vee\neg p)$ both when $R(w)=\varnothing$ and when $R(w)\neq\varnothing$.

\fbox{$\mathbf{S4}$ frames}

Next, we show that $\mathfrak{F}\in\mathbf{S4}$ iff $\blacktriangle p\vdash\blacktriangle\blacktriangle p$ and $\blacktriangle(p\vee\neg p)\vdash p\vee\neg p$ are valid on $\mathfrak{F}$.

Let $\mathfrak{F}$ be transitive and reflexive. Then $\blacktriangle(p\vee\neg p)\vdash p\vee\neg p$ is valid on $\mathfrak{F}$, as we have just shown. It remains to prove that $\blacktriangle p\vdash\blacktriangle\blacktriangle p$ is also valid on $\mathfrak{F}$. Let now $\mathfrak{M}$ be a~model on $\mathfrak{F}$ and let further $\mathfrak{M},w\vDash^+\blacktriangle p$. We show that $\mathfrak{M},w\vDash^+\blacktriangle\blacktriangle p$. Since $R$ is transitive and reflexive, we have three options:
\begin{enumerate}
\item $\mathfrak{M},w'\vDash^+p$ and $\mathfrak{M},w'\nvDash^-p$ for every $w'\in R(w)$, or
\item $\mathfrak{M},w'\nvDash^+p$ and $\mathfrak{M},w'\vDash^-p$ for every $w'\in R(w)$, or
\item $\mathfrak{M},w'\vDash^+p$ and $\mathfrak{M},w'\vDash^-p$ for every $w'\in R(w)$.
\end{enumerate}

In the cases (1) and (2), $\mathfrak{M},w'\vDash^+\blacktriangle p$ and $\mathfrak{M},w'\nvDash^-\blacktriangle p$ in any $w'\in R(w)$ since $R$ is reflexive and transitive. Thus, $\mathfrak{M},w'\vDash^+\blacktriangle\blacktriangle p$ and $\mathfrak{M},w'\nvDash^-\blacktriangle\blacktriangle p$.

In the third case, $\mathfrak{M},w'\vDash^+\blacktriangle p$ and $\mathfrak{M},w'\vDash^-\blacktriangle p$ in any $w'\in R(w)$ since there are no dead-ends. Hence, $\mathfrak{M},w'\vDash^+\blacktriangle\blacktriangle p$ and $\mathfrak{M},w'\vDash^-\blacktriangle\blacktriangle p$.

For the converse, let $\mathfrak{F}\notin\mathbf{S4}$. If $\mathfrak{F}$ is not reflexive, $\blacktriangle(p\vee\neg p)\vdash p\vee\neg p$ is not valid. If $\mathfrak{F}\notin\mathbf{S4}$ but it is reflexive, then it is not transitive.

We let $w_0,w_1,w_2\in\mathfrak{F}$ s.t.\ $w_0Rw_1$ and $w_1Rw_2$ but $w_2\notin R(w_0)$. We define the valuations as follows. $w'\in v^+(p)$ and $w'\notin v^-(p)$ for every $w'\in R(w_0)$; $w''\notin v^+(p)$ $w''\notin v^+(p)$ for every $w''\notin R(w)$. It is clear that $\mathfrak{M},w_0\vDash^+\blacktriangle p$ and $\mathfrak{M},w_0\nvDash^-\blacktriangle p$. However, $\mathfrak{M},w_1\nvDash^+\blacktriangle p$. Now. since $w_0\in R(w_0)$, we have that $\mathfrak{M},w_0\nvDash^+\blacktriangle\blacktriangle p$ and $\mathfrak{M},w_0\vDash^-\blacktriangle\blacktriangle p$, as desired.

\fbox{$\mathbf{S5}$ frames}

Next, we prove that $\mathfrak{F}\in\mathbf{S5}$ iff $\blacktriangledown p\vdash\blacktriangle\blacktriangle p$ and $\blacktriangle(p\vee\neg p)\vdash p\vee\neg p$ are valid on $\mathfrak{F}$.

Assume that $\mathfrak{F}$ is reflexive and Euclidean\footnote{Recall that $R$ is an equivalence relation iff it is reflexive and Euclidean.}. Then $\blacktriangle(p\vee\neg p)\vdash p\vee\neg p$ is valid. Now let $\mathfrak{M}$ be a~model over $\mathfrak{F}$, and let $\mathfrak{M},w_0\vDash^+\blacktriangledown p$. We show that $\mathfrak{M},w_0\vDash^+\blacktriangle\blacktriangle p$. There are three cases:
\begin{enumerate}
\item $\mathfrak{M},w'\vDash^+p$ and $\mathfrak{M},w'\vDash^-p$ at any $w'\in R(w_0)$, or
\item there are $s,t\in R(w_0)$ s.t.\ $\mathfrak{M},s\vDash^+p$ and $\mathfrak{M},t\nvDash^+p$, or
\item there are $s,t\in R(w_0)$ s.t.\ $\mathfrak{M},s\vDash^-p$ and $\mathfrak{M},t\nvDash^-p$.
\end{enumerate}
Note, first of all, that since $R$ is an equivalence relation, it holds that $uRu'$ and $R(u)=R(w_0)$ for any $u,u'\in R(w_0)$.

Now, in the first case, $\mathfrak{M},w'\vDash^+\blacktriangle p$ and $\mathfrak{M},w'\vDash^-\blacktriangle p$ at any $w'\in R(w_0)$. Hence, $\mathfrak{M},w_0\vDash^+\blacktriangle\blacktriangle p$ and $\mathfrak{M},w_0\vDash^-\blacktriangle\blacktriangle p$ as desired.

In the second and in the third cases, $\mathfrak{M},u\nvDash^+\blacktriangle p$ and $\mathfrak{M},u\vDash^-\blacktriangle p$ for any $u\in R(w_0)$. Thus, $\mathfrak{M},w_0\vDash^+\blacktriangle\blacktriangle p$ and $\mathfrak{M},w_0\nvDash^-\blacktriangle\blacktriangle p$.

For the converse, let $\mathfrak{F}\notin\mathbf{S5}$. Again, if $\mathfrak{F}$ is not reflexive, $\blacktriangle(p\vee\neg p)\vdash p\vee\neg p$ will not be valid. So, we consider the case of a~reflexive but non-Euclidean frame $\mathfrak{F}$.

Let $w_0\in\mathfrak{F}$ and also $w_1,w_2\in R(w_0)$ be s.t.\ $w_2\notin R(w_1)$. We set the valuation of $p$ as follows: $w\in v^+(p)$ and $w\notin v^-(p)$ for any $w\in(R(w_0)\setminus R(w_2))\cup(R(w_1)\setminus R(w_2))$ and $w'\in v^+(p)$ and $w'\in v^-(p)$ for any $w'\in R(w_2)$.

Now, $\mathfrak{M},w_0\vDash^+\blacktriangledown p$ and $\mathfrak{M},w_0\nvDash^-\blacktriangledown p$, whence $\mathfrak{M},w_0\nvDash^+\blacktriangle p$ and $\mathfrak{M},w_0\vDash^-\blacktriangle p$. But observe that $\mathfrak{M},w_1\vDash^+\blacktriangle p$ and $\mathfrak{M},w_1\nvDash^-\blacktriangle p$. Thus, since $w_0\in R(w_0)$, we have that $\mathfrak{M},w_0\vDash^-\blacktriangle\blacktriangle p$ and $\mathfrak{M},w_0\nvDash^+\blacktriangle\blacktriangle p$.

\fbox{$\mathbf{F}$ frames}

We show that $\mathfrak{F}\in\mathbf{F}$ iff $\blacktriangledown p\vdash\blacktriangle p$ is valid on $\mathfrak{F}$.

Assume that $\mathfrak{F}$ is partial-functional and let $\mathfrak{M}$ be a~model on $\mathfrak{M}$ s.t.\ $\mathfrak{M},w\vDash^+\blacktriangledown p$. Then there must be exactly one $w'$ s.t.\ $wRw'$. Hence, $\mathfrak{M},w'\vDash^+p$ and $\mathfrak{M},w'\vDash^-p$. But then $\mathfrak{M},w\vDash^+\blacktriangle p$ and $\mathfrak{M},w\vDash^-\blacktriangle p$, as required.

For the converse, assume that $\mathfrak{F}$ is not partial-functional. Thus, there is $w$ s.t.\ $\{w',w''\}\subseteq R(w)$. We set the valuations as follows: $w'\in v^+(p)$, $w'\notin v^-(p)$, $w''\notin v^+(p)$, and $w''\in v^-(p)$. Thus, $\mathfrak{M},w\vDash^+\blacktriangledown p$ and $\mathfrak{M},w\nvDash^+\blacktriangle p$, as required.

\fbox{$\mathbf{Ver}$ frames}

We show that $\mathfrak{F}\in\mathbf{Ver}$ iff $\blacktriangle p$ is valid on $\mathfrak{F}$.

Let $R(w)=\varnothing$ for every $w\in\mathfrak{F}$. Then $\mathfrak{M},w\vDash^+\blacktriangle p$ and $\mathfrak{M},w\nvDash^-\blacktriangle p$. For the converse, let $R(w)\neq\varnothing$ for some $w\in\mathfrak{F}$. We define $w'\notin v^+(p)$ and $w'\notin v^-(p)$ for every $w'\in R(w)$. Thus, $\mathfrak{M},w\nvDash^+\blacktriangle p$.

\fbox{$\mathbf{1}$ frames}

We show that $\mathfrak{F}$ is coreflexive iff $p\vee\neg p\vdash\blacktriangle p$ is valid on $\mathfrak{F}$. Assume that $\mathfrak{F}\in\mathbf{1}$, that $\mathfrak{M}$ is a~model on $\mathfrak{F}$, and that $\mathfrak{M},w\vDash^+p\vee\neg p$. If $R(w)=\varnothing$, $\mathfrak{M},w\vDash^+\blacktriangle p$, as well. If $R(w)=\{w\}$, then we have two cases.
\begin{enumerate}
\item If $\mathfrak{M},w\vDash^-p\vee\neg p$, then $\mathfrak{M},w\vDash^+\blacktriangle p$ and $\mathfrak{M},w\vDash^-\blacktriangle p$.
\item If $\mathfrak{M},w\nvDash^-p\vee\neg p$, then $\mathfrak{M},w\vDash^+\blacktriangle p$ and $\mathfrak{M},w\nvDash^-\blacktriangle p$.
\end{enumerate}

For the converse, assume that $w'\in R(w)$ for some $w\in\mathfrak{F}$ and $w\neq w'$. We set $w\in v^+(p)$, $w\notin v^-(p)$, $w'\notin v^+(p)$, and $w'\notin v^-(p)$. It is clear that $\mathfrak{M},w\vDash^+p\vee\neg p$ but $\mathfrak{M},w\nvDash^+\blacktriangle p$.
\end{proof}
\begin{remark}
It is important to note that $\blacktriangle p\vdash\blacktriangle\blacktriangle p$ \emph{does not define} transitive frames. Neither $\blacktriangledown p\vdash\blacktriangle\blacktriangle p$ defines Euclidean frames.

Indeed, one sees in fig.~\ref{fig:4counterexample} that $\mathfrak{M},w\!\vDash^+\!\blacktriangle p$, $\mathfrak{M},w\!\vDash^-\!\blacktriangle p$, $\mathfrak{M},w'\!\vDash^+\!\blacktriangle p$, and $\mathfrak{M},w'\!\vDash^-\!\blacktriangle p$. But $\mathfrak{M},w''\!\vDash^+\!\blacktriangle p$ and $\mathfrak{M},w''\!\nvDash^-\!\blacktriangle p$. Hence, $\mathfrak{M},w\!\nvDash^+\!\blacktriangle\blacktriangle p$ and $\mathfrak{M},w\!\vDash^-\!\blacktriangle\blacktriangle p$.
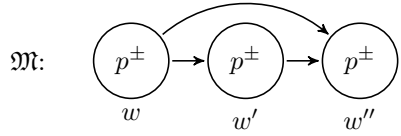
\begin{figure}
\centering
\begin{tikzpicture}[modal,node distance=0.5cm,world/.append style={minimum
size=1cm}]
\node[world] (w) [label=below:{$w$}] {$p^\pm$};
\node[world] (w') [right=of w] [label=below:{$w'$}] {$p^\pm$};
\node[world] (w'') [right=of w'] [label=below:{$w''$}] {$p^\pm$};
\node[] [left=of w] {$\mathfrak{M}$:};
\path[->] (w) edge (w');
\path[->] (w') edge (w'');
\path[->] (w) edge[bend left=40] (w'');
\end{tikzpicture}
\caption{$\blacktriangle p\vdash\blacktriangle\blacktriangle p$ does not define transitive frames.}
\label{fig:4counterexample}
\end{figure}

For the Euclidean frames, consider the counterexample in fig.~\ref{fig:5counterexample}.
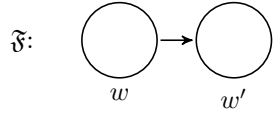
\begin{figure}
\centering
\begin{tikzpicture}[modal,node distance=0.5cm,world/.append style={minimum
size=1cm}]
\node[world] (w) [label=below:{$w$}] {};
\node[world] (w') [right=of w] [label=below:{$w'$}] {};
\node[] [left=of w] {$\mathfrak{F}$:};
\path[->] (w) edge (w');
\end{tikzpicture}
\caption{$\blacktriangledown p\vdash\blacktriangle\blacktriangle p$ does not define Euclidean frames.}
\label{fig:5counterexample}
\end{figure}
It is clear that $\mathfrak{F}$ is not Euclidean. However, since $R(w')=\varnothing$, we have that $\mathfrak{M},w'\vDash^+\blacktriangle p$ and $\mathfrak{M},w'\vDash^-\blacktriangle p$ for any $\mathfrak{M}$ on $\mathfrak{F}$. But then, $\mathfrak{M},w'\vDash^+\blacktriangle\blacktriangle p$ and $\mathfrak{M},w'\vDash^-\blacktriangle\blacktriangle p$ for any $\mathfrak{M}$ on $\mathfrak{F}$. Thus, $\blacktriangledown p\vdash\blacktriangle\blacktriangle p$ is valid on $\mathfrak{F}$.
\end{remark}

We finish the section with one more remark on the differences between $\AFDE$ and $\mathit{LET}_F$.
\begin{remark}\label{rem:LETFcomparison2}
Previously (recall Remark~\ref{rem:LETFcomparison1}), we observed some differences between $\AFDE$ and $\mathit{LET}_F$. However, since $\mathit{LET}_F$ is defined over \emph{partially ordered} frames while $\AFDE$ over all frames, one might rightfully wonder whether these two logics coincide on \emph{partially ordered} frames or $\mathbf{S4}$ frames.

The answer to this is negative. First of all, $\mathfrak{M}'$ seen on fig.~\ref{fig:Trivexample} shows that even on partially ordered frames there are no $\AFDE$ valid formulas (recall that $\circ\phi\vee\bullet\phi$ is $\mathit{LET}_F$ valid). Likewise, while $\circ p\wedge p\wedge\neg p\vdash q$ and $\circ p\wedge\bullet p\vdash q$ are $\mathit{LET}_F$ valid, no sequent of the form $\phi\vdash q$ is $\AFDE$ valid on a partially ordered frame provided that $\phi$ does not contain $q$. Indeed, it is easy to see on model $\mathfrak{M}$ on fig.~\ref{fig:trivialisationcounterexample}.
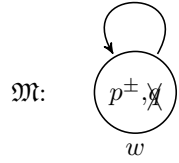
\begin{figure}
\centering
\begin{tikzpicture}[modal,node distance=0.5cm,world/.append style={minimum
size=1cm}]
\node[world] (w) [label=below:{$w$}] {$p^\pm$,$\xcancel{q}$};
\node[] [left=of w] {$\mathfrak{M}$:};
\path[->] (w) edge[reflexive] (w);
\end{tikzpicture}
\caption{The values of all variables in $\phi$ are the same and exemplified by $p$.}
\label{fig:trivialisationcounterexample}
\end{figure}
One can check by induction on $\phi$ that $\mathfrak{M},w\vDash^+\phi$ but $\mathfrak{M},w\nvDash^+q$.
\end{remark}
\section{Concluding remarks\label{sec:conclusion}}
In this paper, we have presented an expansion of First Degree Entailment with a~non-contingency modality~$\blacktriangle$. We interpreted this modality using Kripke semantics as ‘$\phi$ has the same truth value in all accessible states’. Let us briefly recapitulate the main results of the paper.
\begin{itemize}
\item In section~\ref{sec:proofsystem}, we presented an analytical cut calculus for $\AFDE$ and proved its soundness and completeness (theorems~\ref{AFDEsoundness} and~\ref{AFDEcompleteness}).
\item In section~\ref{sec:expressivity}, we proved that $\blacktriangle$ cannot be defined in terms of $\Box$ from $\mathbf{K_{FDE}}$, nor can $\Box$ be defined in terms of $\blacktriangle$ (theorems~\ref{theorem:trianglenotdefinable} and~\ref{theorem:boxnotdefinable}).
\item In section~\ref{sec:framedefinability}, we established (Theorem~\ref{AFDEcorrespondence}) that in contrast to the classical non-contingency logic, we can define several first-order properties on the frames, most importantly, reflexive, preordered ($\mathbf{S4}$), and $\mathbf{S5}$ frames.
\end{itemize}

Still, several questions remain open.

First and foremost, it is unclear whether we can define any other interesting first-order frame properties such as symmetry or seriality. As we have already mentioned above, the negative result for the classical non-contingency logic is due to the fact that classical $\triangle$-formulas cannot distinguish between different functional frames. These are, however, distinguishable by $\Ltriangle$ sequents. Likewise, it is unclear which classes of frames definable in $\mathbf{K_{FDE}}$ are not definable in $\AFDE$.

Second, we have devised an analytic cut calculus in this paper. One of the advantages of these calculi is that they can be used to provide a~decision procedure and complexity evaluation for the logic. Thus, it would be instructive to establish the complexity of $\AFDE$ validity. In fact, to the best of our knowledge, there are no results establishing the complexity of the \emph{classical} non-contingency logics. So, it makes sense, to compare the complexities of $\AFDE$ and the classical non-contingency logics. A promising starting point would be to expand the approach from~\cite{DAgostinoSolares-Rojas2022} to the case of $\AFDE$.

Third, and related\footnote{We thank an anonymous reviewer for bringing this to our attention.} to the previous point is that $\mathbb{S}(\AFDE)$ can be used to formulate the semantics of $\AFDE$ in terms of ‘imprecise values’ that can be associated to the value-labels. Furthermore, in this way, $\mathbb{S}(\AFDE)$ is a labelled extension of $\mathbf{RE}_{\mathrm{fde}}$~\cite{DAgostino1990,DAgostinoSolares-Rojas2022} in the sense of~\cite{DAgostinoGabbay1994}. We leave a more detailed consideration and motivation of this approach for future work.

Finally, we used the language without implication. There are many implicative expansions of the First Degree Entailment, for example with constructive implications (cf.~\cite{Wansing2008}) which also use frame semantics. There are modal expansions of such logics with $\Box$ (cf., e.g.~\cite{Wansing2004}). However, to the best of our knowledge, there are no expansions of these logics with non-contingency modalities.
\section*{Acknowledgements}
The first author's research was supported by grant ANR JCJC 2019, project PRELAP (ANR-19-CE48-0006).

The authors are grateful to two anonymous reviewers and the handling editor whose comments and suggestions enhanced the presentation and content of the paper.
\end{spacing}
\bibliographystyle{plain}
\bibliography{lit}

\begin{thebibliography}{10}

\bibitem{DAgostino1990}
M.{D'} Agostino.
\newblock {\em Investigations into the complexity of some propositional
  calculi}.
\newblock Oxford University Computing Laboratory, Oxford, 1990.

\bibitem{DAgostino1992}
M.{D'} Agostino.
\newblock Are tableaux an improvement on truth-tables?
\newblock {\em Journal of Logic, Language and Information}, 1(3):235--252,
  1992.

\bibitem{DAgostinoGabbay1994}
M.{D'} Agostino and D.M. Gabbay.
\newblock {A generalization of analytic deduction via labelled deductive
  systems. Part I: Basic substructural logics}.
\newblock {\em Journal of Automated Reasoning}, 13(2):243--281, 1994.

\bibitem{DAgostinoMondadori1994}
M.{D'} Agostino and M.~Mondadori.
\newblock {The Taming of the Cut. Classical Refutations with Analytic Cut}.
\newblock {\em Journal of Logic and Computation}, 4(3):285--319, 06 1994.

\bibitem{DAgostinoSolares-Rojas2022}
M.{D'} Agostino and A.~Solares-Rojas.
\newblock {Towards Tractable Approximations to Many-Valued Logics: the Case of
  First Degree Entailment}.
\newblock In {\em The Logica Yearbook 2021}, pages 57--76. College
  Publications, 2022.

\bibitem{Amerbauer1996}
M.~Amerbauer.
\newblock Cut-free tableau calculi for some propositional normal modal logics.
\newblock {\em Studia Logica}, 57(2--3):359--372, 1996.

\bibitem{AntunesCarnielliKapsnerRodriguez2020}
H.~Antunes, W.~Carnielli, A.~Kapsner, and A.~Rodrigues.
\newblock Kripke-style models for logics of evidence and truth.
\newblock {\em Axioms}, 9(3):100, August 2020.

\bibitem{Belnap1977fourvalued}
N.D. Belnap.
\newblock {A Useful Four-Valued Logic}.
\newblock In J.M. Dunn and G.~Epstein, editors, {\em Modern Uses of
  Multiple-Valued Logic}, pages 5--37, Dordrecht, 1977. Springer Netherlands.

\bibitem{Belnap1977computer}
N.D. Belnap.
\newblock How a computer should think.
\newblock In G.~Ryle, editor, {\em Contemporary aspects of philosophy}, pages
  30--55, 1977.

\bibitem{Blasio2017}
C.~Blasio.
\newblock {Revisitando a l{\'o}gica de Dunn--Belnap}.
\newblock {\em Manuscrito}, 40:99--126, 2017.

\bibitem{CaleiroCarnielliConiglioMarcos2005}
C.~Caleiro, W.~Carnielli, M.~Coniglio, and J.~Marcos.
\newblock {Two's company: “The humbug of many logical values”}.
\newblock In {\em Logica universalis}, pages 169--189. Springer, 2005.

\bibitem{CaleiroMarcelinoRivieccio2018}
C.~Caleiro, S.~Marcelino, and U.~Rivieccio.
\newblock Characterizing finite-valuedness.
\newblock {\em Fuzzy Sets and Systems}, 345:113--125, 2018.

\bibitem{CaleiroMarcosVolpe2015}
C.~Caleiro, J.~Marcos, and M.~Volpe.
\newblock Bivalent semantics, generalized compositionality and analytic
  classic-like tableaux for finite-valued logics.
\newblock {\em Theoretical Computer Science}, 603:84--110, 2015.

\bibitem{Costa-Leite2016}
A.~Costa-Leite.
\newblock Interplays of knowledge and non-contingency.
\newblock {\em {Logic and Logical Philosophy}}, 25(4):521--534, 2016.

\bibitem{Drobyshevich2020}
S.~Drobyshevich.
\newblock {A General Framework for FDE-Based Modal Logics}.
\newblock {\em Studia Logica}, pages 1--26, 2020.

\bibitem{Dunn1976}
J.M. Dunn.
\newblock {Intuitive Semantics for First-Degree Entailments and ‘Coupled
  Trees’}.
\newblock {\em Philosophical Studies: An International Journal for Philosophy
  in the Analytic Tradition}, 29(3):149--168, 1976.

\bibitem{Dunn2000}
J.M. Dunn.
\newblock {Partiality and Its Dual}.
\newblock {\em Studia Logica}, 66(1):5--40, Oct 2000.

\bibitem{Fan2019}
J.~Fan.
\newblock Strong noncontingency: On the modal logics of an operator
  expressively weaker than necessity.
\newblock {\em Notre Dame Journal of Formal Logic}, 60(3):407--435, 08 2019.

\bibitem{FanvanDitmarsch2015}
J.~Fan and H.~van Ditmarsch.
\newblock {Neighborhood Contingency Logic}.
\newblock In Mohua Banerjee and Shankara~Narayanan Krishna, editors, {\em
  {Logic and Its Applications}}, pages 88--99, Berlin, Heidelberg, 2015.
  Springer Berlin Heidelberg.

\bibitem{FanWangvanDitmarsch2015}
J.~Fan, Y.~Wang, and H.~van Ditmarsch.
\newblock {Contingency and knowing whether}.
\newblock {\em The Review of Symbolic Logic}, 8(1):75–107, 2015.

\bibitem{Font1997}
J.M. Font.
\newblock {Belnap's Four-Valued Logic and De Morgan Lattices}.
\newblock {\em Logic Journal of the IGPL}, 5:1--29, 05 1997.

\bibitem{Humberstone1995}
I.L. Humberstone.
\newblock The logic of non-contingency.
\newblock {\em Notre Dame Journal of Formal Logic}, 36(2):214--229, 04 1995.

\bibitem{Humberstone2013}
L.~Humberstone.
\newblock {Zolin and Pizzi: Defining Necessity from Noncontingency}.
\newblock {\em Erkenntnis}, 78:1275--1302, 2013.

\bibitem{Indrzejczak2012}
A.~Indrzejczak.
\newblock Cut-free hypersequent calculus for $\mathbf{S4.3}$.
\newblock {\em Bulletin of the Section of Logic}, 41(1):89--104, 2012.

\bibitem{Kuhn1995}
S.T. Kuhn.
\newblock {Minimal Non-contingency Logic}.
\newblock {\em Notre Dame Journal of Formal Logic}, 36(2):230--234, 04 1995.

\bibitem{Nguyen2001}
L.A. Nguyen.
\newblock {Analytic tableau systems and interpolation for the modal logics KB,
  KDB, K5, KD5}.
\newblock {\em Studia Logica}, 69(1):41--57, 2001.

\bibitem{OdintsovWansing2010}
S.P. Odintsov and H.~Wansing.
\newblock {Modal logics with Belnapian truth values}.
\newblock {\em Journal of Applied Non-Classical Logics}, 20(3):279--301, 2010.

\bibitem{OdintsovWansing2017}
S.P. Odintsov and H.~Wansing.
\newblock {Disentangling FDE-Based Paraconsistent Modal Logics}.
\newblock {\em Studia Logica}, 105(6):1221--1254, 2017.

\bibitem{Priest2008FromIftoIs}
G.~Priest.
\newblock {\em {An Introduction to Non-Classical Logic. From If to Is}}.
\newblock Cambridge University Press, 2nd edition, 2008.

\bibitem{Priest2008}
G.~Priest.
\newblock Many-valued modal logics: a simple approach.
\newblock {\em The Review of Symbolic Logic}, 1(2):190--203, 2008.

\bibitem{RivieccioJungJansana2017}
U.~Rivieccio, A.~Jung, and R.~Jansana.
\newblock {Four-valued modal logic: Kripke semantics and duality}.
\newblock {\em Journal of Logic and Computation}, 27(1):155--199, 2017.

\bibitem{ShramkoWansing2005}
Y.~Shramko and H.~Wansing.
\newblock {Some useful 16-valued logics: How a computer network should think}.
\newblock {\em Journal of Philosophical Logic}, 34(2):121--153, 2005.

\bibitem{Wansing2004}
H.~Wansing.
\newblock Connexive modal logic.
\newblock {\em AiML-2004: Advances in Modal Logic}, pages 367--383, 2004.

\bibitem{Wansing2008}
H.~Wansing.
\newblock Constructive negation, implication, and co-implication.
\newblock {\em Journal of Applied Non-Classical Logics}, 18(2--3):341--364,
  2008.

\bibitem{ZaitsevShramko2004english}
D.V. Zaitsev and Y.V. Shramko.
\newblock {Logical Entailment and Designated values (in Russian)}.
\newblock {\em Logical Investgations}, 11:127--138, 2004.

\bibitem{Zolin1999}
E.E. Zolin.
\newblock {Completeness and Definability in the Logic of Noncontingency}.
\newblock {\em Notre Dame Journal of Formal Logic}, 40(4):533--547, 10 1999.

\bibitem{Zolin2002}
E.E. Zolin.
\newblock {Sequential Reflexive Logics with Noncontingency Operator}.
\newblock {\em Mathematical Notes}, 72(6):784--798, 2002.

\end{thebibliography}
\end{document}